%% file: IN_BR_27.tex
\documentclass[12pt]{amsart}
\input {epsf}
\usepackage{latexsym, amsbsy, amsmath, amsfonts, amssymb, amsthm, amscd, amscd}
\usepackage[all]{xy}
\newcommand{\Z}{\mathbb Z}

\newtheorem{Theorem}{Theorem}[section]

\newtheorem{Proposition}{Proposition}[section]
\newtheorem{Remark}{Remark}[section]

\allowdisplaybreaks[2]
\numberwithin{equation}{section}
\numberwithin{figure}{section}
\title{On the inverse braid monoid}
\author[Vershinin]{V.~V.~Vershinin}
\address{D\'epartement des Sciences Math\'ematiques,
                                     Universit\'e Montpellier II,
Place Eug\'ene Bataillon,
34095 Montpellier cedex 5, France}
\email{ vershini@math.univ-montp2.fr}
\address{ Sobolev Institute of Mathematics, Novosibirsk, 630090,
Russia }
\email{ versh@math.nsc.ru}
\thanks{The author was supported  in part by the
 by CNRS-NSF grant No~17149 and INTAS grant No~03-5-3251.}
\begin{document}
\begin{abstract}
Inverse braid monoid describes a structure on braids where the number of strings
is not fixed. So, some strings of initial $n$  may be deleted.
In the paper we show that many properties and objects based on braid
groups may be extended to the inverse braid monoids.
Namely we prove an inclusion into a  monoid of partial monomorphisms of 
a free group. This gives a solution of the word problem. Another solution
is obtained by an approach similar to that of Garside.
 We give also the analogues of Artin presentation
with two generators and Sergiescu graph-presentations.
\end{abstract}
\maketitle
\tableofcontents
\section{Introduction}

The notion of {\it inverse semigroup} was introduced by V.~V.~Wagner in 1952 \cite{Wag}.
By definition it means that for any element $a$ of a semigroup (monoid)
$M$ there exists a unique element $b$ (which is called {\it inverse}) such that 
\begin{equation}
a = aba
\label{eq:reg_v_n}
\end{equation}
 and 
\begin{equation}
b = bab.
\label{eq:inv}
\end{equation}
The roots of this notion can be seen in the von Neumann regular rings \cite{v_N} 
where only one condition (\ref{eq:reg_v_n}) holds for non necessary unique $b$,
or in the Moore-Penrose pseudoinverse for matrices \cite{Mo}, \cite{Pen}
where both conditions
(\ref{eq:reg_v_n}) and (\ref{eq:inv}) hold (and certain supplementary conditions also).

The typical example of an inverse monoid is a monoid of partial (defined on a subset)
injections of a set. For a finite set this gives us the notion of a
{\it symmetric inverse monoid } $I_n$ which generalises and includes the classical
symmetric group $\Sigma_n$. A presentation of symmetric inverse monoid was
obtained by L.~M.~Popova \cite{Po}, see also formulas 
(\ref{eq:brelations}), (\ref{eq:invbrelations}\,-\ref{eq:syminvrelations})
below.
Recently the {\it inverse braid monoid }  $IB_n$ was constructed by 
D.~Easdown and T.~G.~Lavers \cite{EL}. 
It arises from 
a very natural operation on braids: deleting one or several strings.
By the application of this procedure to braids in
$Br_n$ we get {\it partial} braids  \cite{EL}.
The multiplication of partial braids is shown at the Figure~\ref{fi:vv2}

\begin{figure}
\epsfbox{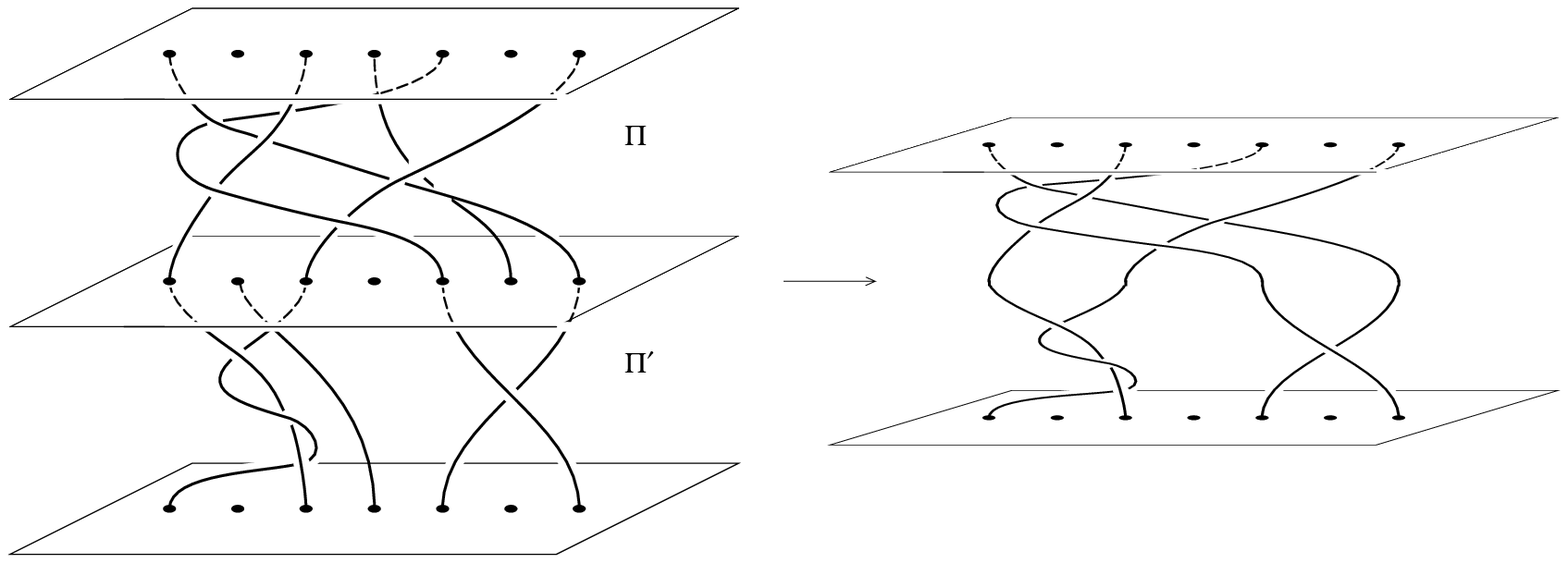}
\caption{} \label{fi:vv2}
\end{figure}

At the last stage it is necessary to remove any arc that does not join the upper 
or lower planes.

 The set of all
partial braids  with this operation forms an inverse braid monoid $IB_n$. 

One of the motivations to study $IB_n$ is that it is a natural setting for
the {\it Makanin braids}, which were also called by {\it smooth} braids by 
G.~S.~Makanin who first mentioned them in
\cite{Kou}, (page 78, question 6.23), and D.~L.~Johnson \cite{Joh1},
and by {\it Brunian} braids in the work of
J.~A.~Berrick, F.~R.~Cohen, Y.~L.~Wong and J.~Wu \cite{BCWW}). 
By the  usual definition a braid is  Makanin if it becomes trivial
after deleting any string, see formulas (\ref{eq:imak} - \ref{eq:mak}).
 According to the works of Fred Cohen, Jon Berrick, 
Wu Jie and others Makanin braids have connections with homotopy groups
of spheres. Namely the exists an exact sequence
\begin{equation}
1\to Mak_{n+1}(S^2) \to Mak_{n}(D^2) \to Mak_{n}(S^2) \to 
\pi_{n-1}(S^2)\to 1
\label{eq:maks2}
\end{equation} 
for  $n\geq 5$,
where $Mak_{n}(D^2) $ is the group Makanin braids  and 
$Mak_{n}(S^2) $ is the group of Makanin braids of the sphere $S^2$,
see Section~\ref{sec:gen}.

The purpose of this paper is to demonstrate that canonical properties of 
braid groups and notions based on braids often have there smooth continuation 
for the inverse braid monoid $IB_n$.

Usually  the braid group $Br_n$ 
is given by the following Artin presentation \cite{Art1}.
It has the generators $\sigma_i$, 
$i=1, ..., n-1$ and two types of relations: 
\begin{equation}
 \begin{cases} \sigma_i \sigma_j &=\sigma_j \, \sigma_i, \ \
\text{if} \ \ |i-j|
>1,
\\ \sigma_i \sigma_{i+1} \sigma_i &= \sigma_{i+1} \sigma_i \sigma_{i+1}
\end{cases} \label{eq:brelations}
\end{equation}
Classical braid groups $Br_n$ can be defined also as 
the mapping class group of a disc $D^2$ with $n$ points deleted (or fixed) and 
with its boundary fixed, or as the subgroup of the
automorphism group of a free group 
 $\operatorname{Aut} F_n, $ 
generated by the following automorphisms:
\begin{equation} \begin{cases} 
x_i &\mapsto x_{i+1},
\\ x_{i+1} &\mapsto x_{i+1}^{-1}x_ix_{i+1}, \\
x_j &\mapsto x_j, j\not=i,i+1. 
\end{cases} \label{eq:autf}
\end{equation}
 Geometrically
this action is depicted in the Figure~\ref{fi:mapcl},
where
$x_i$ correspond to the canonical loops on $D^2$ which form the
generators of the fundamental group the punctured disc.
\begin{figure}
\input mapcl
\caption{}\label{fi:mapcl}
\end{figure} 

There exist other presentations of the braid group.
Let 
\begin{equation*}
\sigma = \sigma_1 \sigma_{2} \dots \sigma_{n-1},
\label{eq:sigma}
\end{equation*} 
then the group $Br_n$ is generated by $\sigma_1$ and $\sigma$ because
\begin{equation*}
\sigma_{i+1} =\sigma^i \sigma_1 \sigma^{-i}, \quad i =1, \dots
{n-2}.
\label{eq:sigma_i}
\end{equation*} 
The relations for the generators $\sigma_1$ and $\sigma$ are the 
following
\begin{equation}
 \begin{cases}
\sigma_1 \sigma^i \sigma_1 \sigma^{-i} &= 
\sigma^i \sigma_1 \sigma^{-i} \sigma_1 \ \  \text{for} \ \
2 \leq i\leq {n / 2}, \\
\sigma^n &= (\sigma \sigma_1)^{n-1}.
\end{cases} \label{eq:2relations}
\end{equation}
The presentation (\ref{eq:2relations}) was given by Artin in the initial 
paper \cite{Art1}.
This presentation was also mentioned in the books by F.~Klein \cite{Kl}
and by H.~S.~M.~Coxeter and W.~O.~J.~Moser \cite{CM}.

An interesting series of presentations was given by V.~Sergiescu
\cite{Ser}. For every planar graph he constructed a presentation of the
group $Br_n$, where $n$ is the number of vertices of the graph,
with generators corresponding to edges and relations reflecting
the geometry of the graph. 
To each edge $e$ of  the graph he associates the braid $\sigma_e$ which is a 
clockwise half-twist along $e$ (see Figure~\ref{fig:edges}). 
Artin's classical presentation (\ref{eq:brelations}) in this context
corresponds to the graph consisting of the interval from 1 to $n$
with the natural numbers (from 1 to $n$) as vertices and with
segments between them as edges. 

\begin{figure}
\epsfbox{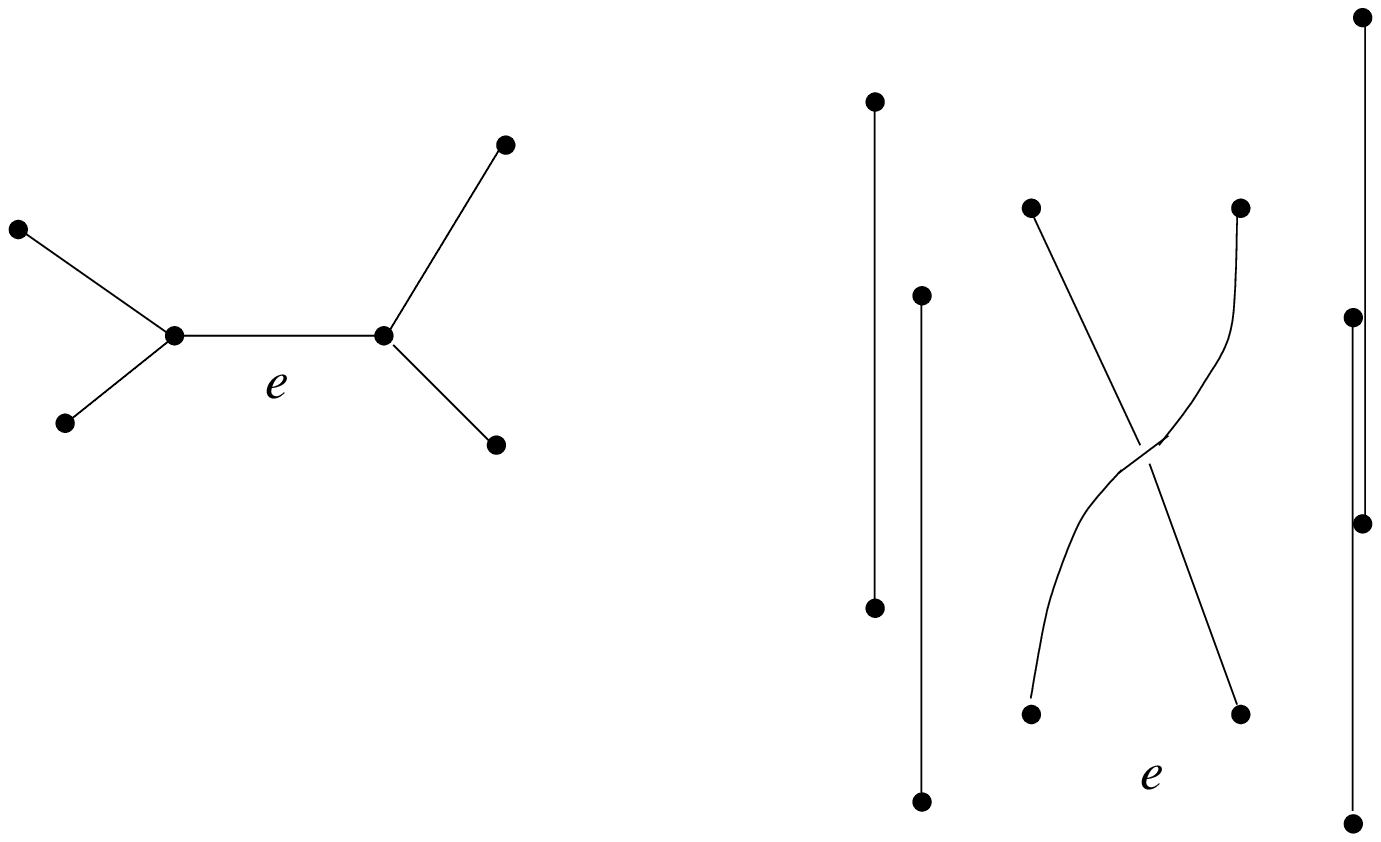}
\caption{} \label{fig:edges}
\end{figure}


Let $\vert \, \, \vert : \Sigma_n\to \Z$
be the length function on the symmetric group with respect to the 
standard generators $s_i$: for $x\in \Sigma_n$,
$\vert  x \vert $ is the smallest natural number $k$ such that $x$ 
is a product of $k$ elements of the set 
$\{s_1,...,s_{n-1}\}$.  
It is known (\cite{Bo}, Sect.~1, Ex.~13(b)) that two minimal expressions
for an element of $\Sigma_n$ are equivalent by using only the 
relations (\ref{eq:brelations}). This implies that
the canonical projection $\tau_n :Br_n\to \Sigma_n$ has a  unique  
set-theoretic section
$r:\Sigma_n\to Br_n$ 
such that $r(s_i)=\sigma_i$ for $i=1,...,n-1$
and $r(xy)= r(x)\, r(y)$ whenever 
$\vert  xy\vert =\vert  x\vert +\vert  y\vert$. 
The image  $r(\Sigma_n)$ under the name of {\it positive permutation braids}
 was studied by 
E.~El-Rifai and H.~R.~Morton \cite{EM}.

The following presentation for the inverse braid monoid was obtained in
\cite{EL}. It has the generators $\sigma_i, \sigma_i^{-1} $, $i=1,\dots,n-1,$
$\epsilon$, and relations
\begin{equation}
 \begin{cases} 
&\sigma_i\sigma_i^{-1}=\sigma_i^{-1}\sigma_i =1, \ \text {for \ all} \ i, \\
&\epsilon \sigma_i  =\, \sigma_i \epsilon \ \ \text {for } i\geq 2,   \\ 
&\epsilon\sigma_1 \epsilon  = \sigma_{1} \epsilon \sigma_1 \epsilon = 
\epsilon\sigma_{1} \epsilon \sigma_1, \\
&\epsilon = \epsilon^2 = \epsilon \sigma_1^2= \sigma_1^2 \epsilon
\end{cases} \label{eq:invbrelations}
\end{equation}
and the braid relations (\ref{eq:brelations}).

Geometrically the generator $\epsilon$ means that the first string in the trivial 
braid is absent.
 
If we replace the first relation in (\ref{eq:invbrelations})
 by the following set of relations
\begin{equation}
\sigma_i^2 =1, \ \text {for \ all} \ i, \\
 \label{eq:syminvrelations}
\end{equation}
and delete the
superfluous relations 
$$\epsilon =  \epsilon \sigma_1^2= \sigma_1^2 \epsilon, $$
we get a presentation of the symmetric inverse monoid $I_n$ \cite{Po} . 
We also can simply add the relations (\ref{eq:syminvrelations})
if we don't worry about redundant relations.
We get a canonical map \cite{EL}
\begin{equation}
\tau_n: IB_n\to I_n
 \label{eq:tauIBn}
\end{equation}
which is a natural extension of the corresponding map for the braid 
and symmetric groups.

More balanced relations for the inverse braid monoid were 
obtained in \cite{Gil}.
Let $\epsilon_i$ denote the trivial braid with $i$-th string deleted,
formally:
\begin{equation*}
\begin{cases} \epsilon_1 &= \epsilon, \\
\epsilon_{i+1} &= \sigma_i^{\pm 1}\epsilon_{i}\sigma_i^{\pm 1}.  
\end{cases} \end{equation*} 
So, the generators are: $\sigma_i, \sigma_i^{-1} $, $i=1,\dots,n-1,$
$\epsilon_i$, $i=1,\dots,n$, and relations  are the following:

\begin{equation}
 \begin{cases} 
&\sigma_i\sigma_i^{-1}=\sigma_i^{-1}\sigma_i =1, \ \text {for \ all} \ i, \\
&\epsilon_j \sigma_i  =\, \sigma_i \epsilon_j \ \ \text {for } \ |j-i|>1, \\ 
&\epsilon_i\sigma_i =  \sigma_{i} \epsilon_{i+1},  \\
&\epsilon_{i+1}\sigma_i =  \sigma_{i} \epsilon_{i},  \\
&\epsilon_i = \epsilon_i^2 , \\
& \epsilon_{i+1} \sigma_i^2= \sigma_i^2 \epsilon_{i+1} = \epsilon_{i+1}, \\
&\epsilon_i \epsilon_{i+1} \sigma_i = \sigma_{i} \epsilon_i \epsilon_{i+1}
=\epsilon_i\epsilon_{i+1},
\end{cases} \label{eq:invbrelations2}
\end{equation}
plus the braid relations (\ref{eq:brelations}).

\section{Properties of inverse braid monoid\label{sec:prop}}

The relations  (\ref{eq:invbrelations}) look asymmetric: one generator
for the idempotent part and $n-1$ generators for the group part.
If we minimise the number of generators of the group part and take the 
presentation  (\ref{eq:2relations}) for the braid group we get a presentation
of the inverse braid monoid with generators $\sigma_1, \sigma$,  $\epsilon$,
and relations:
\begin{equation}
 \begin{cases} 
&\sigma_1\sigma_1^{-1}=\sigma_1^{-1}\sigma_1 =1,  \\
&\sigma\sigma^{-1}=\sigma^{-1}\sigma =1,  \\
&\epsilon \sigma^{i}\sigma_1\sigma^{-i}  
=\, \sigma^{i}\sigma_1\sigma^{-i} \epsilon \ \ \text {for } 1\leq i \leq n-2,   \\ 
&\epsilon\sigma_1 \epsilon  = \sigma_{1} \epsilon \sigma_1 \epsilon = 
\epsilon\sigma_{1} \epsilon \sigma_1, \\
&\epsilon = \epsilon^2 = \epsilon \sigma_1^2= \sigma_1^2 \epsilon,
\end{cases} \label{eq:3invbrelations}
\end{equation}
plus (\ref{eq:2relations}).

Let $\Gamma$ be a planar graph of the Sergiescu graph presentation of the braid 
group \cite{Ser}, \cite{BelV}. Let us add new generators $\epsilon_v$ which 
correspond to each vertex of the graph $\Gamma$. Geometrically it means the
absence in the trivial braid of one string corresponding to the vertex $v$.
We orient the graph $\Gamma$ arbitrarily and so we get a starting $v_0=v_0(e)$
and a terminal $v_1=v_1(e)$ vertex for each edge $e$.
Consider the following relations
\begin{equation}
 \begin{cases} 
&\sigma_e\sigma_e^{-1}=\sigma_e^{-1}\sigma_e =1, \ \text {for  all  edges  of} \  
\Gamma, \\
&\epsilon_v \sigma_e  =\, \sigma_e \epsilon_v, \ \ \text {if  the vertex } \ 
v  \ \text{and  the  edge} \ e \ \text{do  not  intersect}, \\ 
&\epsilon_{v_0}\sigma_e =  \sigma_{e} \epsilon_{v_1}, \ \text{where} \
v_0=v_0(e), \ v_1= v_1(e), \\
&\epsilon_{v_1}\sigma_e =  \sigma_{e} \epsilon_{v_0},  \\
&\epsilon_v = \epsilon_e^2 , \\
& \epsilon_{v_i} \sigma_e^2= \sigma_e^2 \epsilon_{v_i} = \epsilon_{v_i}, 
\ \ i=0,1, \\
&\epsilon_{v_0}\epsilon_{v_1} \sigma_e = \sigma_{e} \epsilon_{v_0} \epsilon_{v_1}
=\epsilon_{v_0}\epsilon_{v_1}.
\end{cases} \label{eq:invbserel}
\end{equation}

\begin{Theorem} We get a Sergiescu graph presentation of the inverse braid monoid 
$IB_n$ if we add to the graph presentation of the braid group $Br_n$ the 
 relations (\ref{eq:invbserel}).
\end{Theorem} 
\hfill $\square$

 A {\it positive partial braid } is a element of $IB_n$
which can be written as a word with only positive entries of the generators
$\sigma_i$, $i=1, \dots, n-1$.

A positive  partial braid is called a {\it positive partial 
permutation braid } if it can be drawn as a geometric positive partial braid
in which every pair of strings crosses at most once.

Write $IB_n^+$ for the set of positive partial 
permutation braids. 
\begin{Proposition} If the partial braids $b_1$, $b_2\in IB_n^+$ induce the same
partial permutation on their strings, then $b_1=b_2$. For each $s\in I_n$ 
there is a partial braid $b\in IB_n^+$, which induces this partial permutation:
$\tau(b)=s$.
\end{Proposition}
\begin{proof} The original arguments for $Br_n$ are geometrical and 
so they translate completely to the case of  partial braids. 
\end{proof}

Let $E F_n$ be a monoid of partial isomorphisms of a free group
$F_n$ defined as follows. Let $a$ be an element of the symmetric
inverse monoid $I_n$, $a\in I_n$, $J_k =\{j_1, \dots, j_k\}$ 
is the image of $a$, and elements $i_1, \dots, i_k$ belong to
domain of the definition of $a$. The monoid $E F_n$
consists of isomorphisms
$$<x_{i_1}, \dots, x_{i_k}> \, \to \ <x_{j_1}, \dots, x_{j_k}>$$
expressed by 
$$f_a :x_i\mapsto w_i^{-1} x_{a(i)}w_i $$
if $i$ is among $i_1, \dots, i_k$ and not defined otherwise and 
$w_i$ is a word on $x_{j_1}, \dots, x_{j_k}$.
The composition of $f_a$ and $g_b$, $a,b\in I_n$ 
is defined for $x_i$ belonging to the domain of $a\circ b$.
 We define a map $\phi_n$ from $IB_n$ to $E F_n$
expanding the canonical inclusion 
\begin{equation*}
Br_n \to  \operatorname{Aut} F_n
 \end{equation*}
by the condition that $\phi_n(\epsilon)$ 
as a partial isomorphism of $F_n$ is given by the
formula
\begin{equation} 
\phi(\epsilon)(x_i) = \begin{cases}
x_i {\text{ if} } \ i\geq 2 , \\ 
{\text {not defined,  if }} i=1 . 
\end{cases} \label{eq:endf1}
\end{equation}

Using the presentation (\ref{eq:invbrelations}) we see that $\phi_n$ is 
correctly defined homomorphism of monoids
\begin{equation*}
\phi_n: IB_n \to  E F_n.
 \end{equation*}
 
\begin{Theorem} The homomorphism $\phi_n$ is a monomorphism.
\label{theo:isoendo} 
\end{Theorem} 
\begin{proof}
Monoid $IB_n$ as a set is a disjoint union of copies of braid groups
$Br_{k}$, $k= 0, \dots n$. (See \cite{Gil} for the exact formula of 
this splitting of $IB_n$ as a groupoid.) Each copy of the group $Br_k$
is identified by the numbers of inputs of strings 
$i_1, \dots, i_{k}$ and outputs of them $j_1, \dots, j_{k}$. 
Let $I_k = \{i_1, i_2, \dots, i_{k}\}$, $i_1< i_2< \dots <i_{k}$,
$J_k = \{j_1, j_2,  \dots, j_{k}\}$, $j_1<j_2< \dots <j_{k}$,
and let $Br(I_k,J_k)$ be the corresponding copy of the braid group.
So
\begin{equation}
IB_n = \amalg_{I_k, J_k \subset \{1,\dots n\}} Br(I_k,J_k).
\end{equation}
Define a homomorphism 
$$\psi(I_k, J_k): Br_n \to E F_n.$$
Let $\gamma(I_k)$ be the homomorphism $F_n \to F_k$ defined by
\begin{equation} \begin{cases} 
x_{i_l} &\mapsto x_{l},  \\ 
x_{s} &\mapsto e \ \text{ if} \ s \not\in I_k.
\end{cases} \label{eq:endfs}
\end{equation}
Homomorphism 
$$\beta(J_k): F_k\to F_n$$ 
we define as an inclusion
$$\beta(J_k)(x_l) = x_{j_l}, \ \ l=1, \dots , k.$$
For each automorphism $\alpha: F_k\to F_k$, $\alpha\in Br_n$,
 its image 
$\psi(I_k,J_k)(\alpha) $ in 
$E F_n$ is defined as a composition
$$\psi(I_k,J_k)(\alpha) = \beta(J_k) \, \alpha \, \gamma(I_k),$$
we compose from right to left as for functions.
Homomorphism $\psi(I_k, J_k)$ is a monomorphism.
Consider the following  diagram
\begin{equation}
\begin{CD}
Br_k@>Id>>Br_k   \\
@VV\rho V  @VV\psi(I_k,J_k)V \\
Br(I_k,J_k)@>\phi_n>> E F_n
\label{eq:comd} \end{CD}
\end{equation}
where  the left hand map
$\rho$
is the bijection. Let us prove that the diagram commutes.
Consider a generator of $Br_k$, say $\sigma_1$. We denote
$\rho(\sigma_1)$ by $\sigma(i_1,i_2; j_1,j_2) \in IB_n$. This is the 
positive partial braid where the string starting at $i_1$ goes to
$j_1$ and the string starting at $i_2$ goes to $j_2$. There is no
strings starting before $i_1$, between $i_1$ and $i_2$,
ending before $j_1$ and between $j_1$ and $j_2$.
Suppose that $i_1<j_2<i_2<j_1$, the other cases can be 
considered the same way. The partial braid 
$\sigma(i_1,i_2; j_1,j_2) \in IB_n$ as an element of the 
inverse braid monoid can be expressed as a word on generators
in the following form:
\begin{equation*}
\sigma(i_1,i_2; j_1,j_2) =
\sigma_{i_1} \sigma_{i_1+1} \dots \sigma_{i_2} \dots 
\sigma_{j_1-1} \sigma_{i_2-2} \dots \sigma_{j_1}
\epsilon_{i_1} \epsilon_{i_1+1} \dots \epsilon_{j_2-1} 
\epsilon_{j_2+1} \dots \epsilon_{j_1-1}. 
\end{equation*}
Note that the expression $\sigma_{i_2-2} \dots \sigma_{j_1}$ is present 
in the formula only if $i_2-2\geq j_2$.
We denote it also as consisting of the two parts:
\begin{equation*}
\sigma(i_1,i_2; j_1,j_2) =
\sigma \epsilon. 
\end{equation*}
Let us study the action of $\sigma(i_1,i_2; j_1,j_2)$ on the 
generators of the free group. We have:
\begin{equation*}
\sigma(x_{i_1}) = x_{j_2} 
\end{equation*}
and then apply the action of the part $\epsilon$:
\begin{equation*}
\epsilon(x_{j_2}) = x_{j_2}.
\end{equation*}
Also we have:
\begin{equation*}
\sigma(x_{l}) = x_{j_2}^{-1}x_{l} x_{j_2}  \ \ \text{for} \ \ 
i_1 < l < i_2.
\end{equation*}
After the application of $\epsilon$
we obtain $\sigma(i_1,i_2; j_1,j_2)(x_l) =e $.
We have
\begin{equation*}
\sigma(x_{i_2}) = x_{j_2}^{-1}x_{i_2-1}^{-1}\dots x_{j_1+1}^{-1}
x_{j_1} x_{j_1+1} \dots x_{i_2-1} x_{j_2} 
\end{equation*}
and then apply the action of the part $\epsilon$:
\begin{equation*}
\epsilon(x_{j_2}^{-1}x_{i_2-1}^{-1}\dots x_{j_1+1}^{-1}
x_{j_1} x_{j_1+1} \dots x_{i_2-1} x_{j_2} )=
x_{j_2}^{-1}x_{j_1} x_{j_2}. 
\end{equation*}
We get exactly the action of the image of $\sigma_1$ by the 
composition of the canonical inclusion and the map $\psi(I_k, J_k)$.
The diagram (\ref{eq:comd}) commutes.
So, $\phi_n$ is also a monomorphism.
The different copies of  $Br(I_k,J_k)$ of $IB_n$ do not
intersect in  $E F_n$. So,
\begin{equation*}
\phi_n: IB_n \to E F_n
 \end{equation*}
 is a monomorphism.
\end{proof}
\begin{Theorem} The monomorphism $\phi_n$ gives a 
solution of the word problem for the inverse braid monoid
in the presentations (\ref{eq:brelations}), (\ref{eq:invbrelations}),
(\ref{eq:invbrelations2}), (\ref{eq:invbserel}) and (\ref{eq:3invbrelations}).
\end{Theorem} 
\begin{proof} As for the braid group if follows from the fact that two words
represent the same element of the monoid iff they have the same action
on the finite set of generators of the free group $F_n$.
\end{proof}

Theorem~\ref{theo:isoendo}  gives also a possibility to interpret the inverse 
braid monoid as
a monoid of isotopy classes of maps. As usual consider a disc $D^2$ with $n$ 
fixed points. Denote the set of these points by $Q_n$.
The fundamental group of $D^2$ with these points deleted is isomorphic to $F_n$. 
Consider homeomorphisms
 of $D^2$ onto a copy of the same disc with the condition that
only $k$ points of $Q_n$,  $k \leq n$ (say $i_1, \dots, i_k$) are mapped
bijectively onto the $k$ points (say $j_1, \dots, j_k$) of the second copy 
of $D^2$. 
Consider the isotopy 
classes of such inclusions and denote the set of them by $IM_n(D^2)$. Evidently 
it is a monoid.

\begin{Theorem} The monoids $IB_n$ and  $IM_n(D^2)$
are isomorphic.
\end{Theorem} 
\begin{proof} 
The same way as in the proof for the braid group using Alexander's trick
we associate a partial braid 
to an element of $IM_n(D^2)$ and prove that it is an isomorphism. 
\end{proof}

These considerations can be generalized to the following definition.
Consider a surface $S_{g,b,n}$ of the genus $g$ with $b$ boundary components and
the set $Q_n$ of $n$ fixed points. Let $f$ be a homeomorphism of $S_{g,b,n}$ 
which maps  $k$ points,  $k \leq n$, from $Q_n$:  $\{i_1, \dots, i_k\}$ to $k$ points 
$\{j_1, \dots, j_k\}$ also from $Q_n$. The same way let $h$  
be a homeomorphism of $S_{g,b,n}$ 
which maps  $l$ points, $l\leq n$, from $Q_n$, say $\{s_1, \dots, s_l\}$ to $l$ points 
$\{t_1, \dots, t_l\}$ again from $Q_n$. Consider the intersection of the sets
$\{j_1, \dots, j_k\}$ and $\{s_1, \dots, s_l\}$,  let it be the set of cardinality $m$,
it may be empty. Then the composition of $f$ and $h$ maps $m$ points of $Q_n$
to $m$ points (may be different) of $Q_n$. If $m=0$ then the composition have
no relation to the set $Q_n$. Denote the set of isotopy classes of such maps
by $\mathcal I \mathcal {M}_{g,b,n}$. Composition defines a structure of monoid
on $\mathcal I \mathcal {M}_{g,b,n}$.
\begin{Proposition} The monoid
$\mathcal I \mathcal {M}_{g,b,n}$ is inverse.
\end{Proposition}
\begin{proof} Each element of $\mathcal I \mathcal {M}_{g,b,n}$ is represented 
by a homeomorphism $h$ of $S_{g,b,n}$. So, take an inverse of $h$ and get the 
identities (\ref{eq:reg_v_n}) and (\ref{eq:inv}).
\end{proof}
We call the monoid $\mathcal I \mathcal {M}_{g,b,n}$ the {\it inverse mapping class monoid}.
If $g=0 $ and $b=1$ we get the inverse braid monoid. In 
the general case $\mathcal I \mathcal {M}_{g,b,n}$ the role of the empty braid plays the
mapping class group  $ \mathcal {M}_{g,b}$ (without fixed points). 

We remind that a monoid $M$ is {\it factorisable} if $M= EG$ where $E $ is a set of 
idempotents of $M$ and $G$ is a subgroup of $M$. 

\begin{Proposition} The monoid
$\mathcal I \mathcal {M}_{g,b,n}$ 
can be written in the form
$$\mathcal I \mathcal {M}_{g,b,n} = E \mathcal {M}_{g,b,n},$$
where $E $ is a set of 
idempotents of $\mathcal I \mathcal {M}_{g,b,n}$  and $\mathcal {M}_{g,b,n}$ is 
the corresponding mapping class group. So this monoid is factorisable.
\end{Proposition}
\begin{proof} An  element of $\mathcal I \mathcal {M}_{g,b,n}$ is represented 
by a homeomorphism $h$ of $S_{g,b,n}$ which maps  $k$ points,  $k \leq n$, from $Q_n$:  $\{i_1, \dots, i_k\}$ to $k$ points 
$\{j_1, \dots, j_k\}$ from $Q_n$. In the isotopy class of $h$ we find a homeomorphism
$h_1$ which maps arbitrarily $Q_n \setminus \{i_1, \dots, i_k\}$ to 
$Q_n \setminus \{j_1, \dots, j_k\}$. Necessary idempotent element 
is the isotopy class of the identity
homeomorphism which fixes only the points $\{i_1, \dots, i_k\}$. 
\end{proof}

Let $\Delta$ be the Garside's   fundamental word   in the braid 
group $Br_{n}$ \cite{Gar}. It can be defined by the formula:
$$\Delta = \sigma_1 \dots \sigma_{n-1} \sigma_1 \dots \sigma_{n-2} \dots  
\sigma_1 \sigma_2 \sigma_1.$$
If we use Garside's notation $\Pi_t\equiv \sigma_1\dots \sigma_t$, then
$\Delta \equiv \Pi_{n-1} \dots \Pi_1$.
\begin{Proposition} The generators $\epsilon_i$ commute with $\Delta$ in
the following way:
\begin{equation*}
\epsilon_i\Delta = \Delta \epsilon_{n+1-i}.
\end{equation*}
\end{Proposition}
\begin{proof} Direct calculation using the second, third and the forth relations 
in (\ref{eq:invbrelations2}).
\end{proof}
\begin{Proposition} The center of $IB_n$ consists of the union of the center of the
braid group $Br_n$ (generated by $\Delta^2$) and the empty braid 
$\varnothing = \epsilon_1 \dots \epsilon_n$.
\end{Proposition}
\begin{proof} The given element lie in the center. Suppose that there are 
other ones. Let $c$ be one of them. It is a partial braid with starting points
$I_k =\{i_1,\dots, i_k\}$ and ending points $J_k =\{j_1,\dots, j_k\}$. 
Take  the one-string partial braid $x$ that starts in the complement
of $J_k$ and ends in $I_k$.  Then $cx$ is the empty braid, while $xc$ is not. 
\end{proof}
Let $\mathcal E$ be the monoid generated by one idempotent generator $\epsilon$ . 
\begin{Proposition} The abelianisation of  $IB_n$ is isomorphic to
$\mathcal E \oplus \Z$. The canonical map
\begin{equation*}
a: IB_n \to \mathcal E \oplus \Z
\end{equation*}
is given by the formula:
\begin{equation*}
\begin{cases}
a(\epsilon_i) = \epsilon ,\\
a(\sigma_i) = 1 .
\end{cases}
\end{equation*}
\end{Proposition}
\hfill $\square$

Let $\epsilon_{k+1, n}$ denote the partial braid with the trivial first $k$ strings 
and the absent rest $n-k$ strings. It can be expressed using the generator $\epsilon$ 
or the generators $\epsilon_i$ as follows
\begin{equation} \epsilon_{k+1, n}=
\epsilon\sigma_{n-1}\dots\sigma_{k+1}\epsilon \sigma_{n-1}\dots\sigma_{k+2}
\epsilon\dots \epsilon\sigma_{n-1}\sigma_{n-2}\epsilon \sigma_{n-1}\epsilon,
\end{equation}
\begin{equation} \epsilon_{k+1, n}=
\epsilon_{k+1}\epsilon_{k+2} \dots \epsilon_{n},
 \label{eq:espki}
\end{equation}
It was proved in \cite{EL} the 
every partial braid has a representative of the form
\begin{equation} \sigma_{i_1}\dots\sigma_{1}\dots \sigma_{i_k}\dots\sigma_{k}
\epsilon_{k+1, n}x \epsilon_{k+1, n}\sigma_{k}\dots\sigma_{j_k}\dots\sigma_{1}
\dots\sigma_{j_1},
 \\ 
\label{eq:form_inv}\end{equation}
\begin{equation} k\in \{0,\dots, n\}, x\in Br_k,
0\leq i_1<\dots<i_k\leq n-1  \ 
\text{and} \  0\leq j_1<\dots<j_k\leq n-1.
\end{equation}
Note that in the formula (\ref{eq:form_inv}) we can use delete one of the 
$\epsilon_{k+1,n}$, but we shall use the form (\ref{eq:form_inv}) because of 
convenience: two symbols $\epsilon_{k+1,n}$ serve as markers to distinguish
the elements of $Br_k$.
We can put the element $x\in Br_k$ in the Markov normal form  \cite{Mar2}
and get the corresponding {\it Markov normal form for the inverse braid monoid} 
$IB_n$. The same way for the Garside normal form.

Let us remind the mains point of Garside's construction. 
Essential role in Garside work  plays the monoid of 
\emph{positive} braids $Br_n^+$, that is the monoid which has a presentation
with generators  $\sigma_i$, 
$i=1, ..., n$ and relations (\ref{eq:brelations}). In other words each
element of this monoid can be represented as a word on the elements 
$\sigma_i$, $i=1, ..., n$ with no entrances of  $\sigma_i^{-1}$.
Two positive words $V$ and $W$ in the
alphabet $\{\sigma_i$, $(i=1,\dots,n-1) \}$
will be said to be {\it positively equal} if they are equal as elements
of $Br_n^+$. Usually this is written as $V\doteq W$.

Among positive words on the alphabet  $\{\sigma_1 \dots \sigma_n\}$ let us 
introduce a lexicographical ordering with the condition that  
$\sigma_1 < \sigma_2 < \dots < \sigma_n $. For a positive word $V$ the 
\emph{base} of $V$ is the smallest positive word which is positively equal
to $V$. The base is uniquely determined. If a positive word $V$ is prime 
to $\Delta$, then for the base of $V$ the notation $\overline{V}$  will
be used.
\begin{Theorem}
 Every word $W$ in $IBr_{n}$ can be uniquely written in
the form 
\begin{equation} \sigma_{i_1}\dots\sigma_{1}\dots \sigma_{i_k}\dots\sigma_{k}
\epsilon_{k+1, n}x \epsilon_{k+1, n}\sigma_{k}\dots\sigma_{j_k}\dots\sigma_{1}
\dots\sigma_{j_1},
 \\ 
\end{equation}
\begin{equation} k\in \{0,\dots, n\}, x\in Br_k,
0\leq i_1<\dots<i_k\leq n-1  \ 
\text{and} \  0\leq j_1<\dots<j_k\leq n-1.
\end{equation} 
where $x$ is written in the Garside normal form for $Br_k$
$$\Delta^m \overline{V},$$
where $m$ is an  integer.
\label{Theorem:garnfm}
\end{Theorem}
\begin{proof}
Note that the elements $\sigma_{i_1}\dots\sigma_{1}\dots \sigma_{i_k}
\dots\sigma_{k}$ and $ \sigma_{k}\dots\sigma_{j_k}\dots\sigma_{1}
\dots\sigma_{j_1}$ are uniquely determined by a given element of $IB_n$ 
(written as a word $W$ in the alphabet $A= \{\sigma_i, \sigma_i^{-1} $, 
$i=1,\dots,n-1,$ $\epsilon\}$). Then Theorem follows from the existence of 
the Garside normal form for $Br_k$.
\end{proof}
Theorem~\ref{Theorem:garnfm} is evidently true also for the presentation with 
$\epsilon_i$,
$i= 1,\dots n$. In this case the elements $\epsilon_{k+1,n}$ are expressed
by (\ref{eq:espki}). 

The form of a word $W$ established in this theorem we call the
\emph{Garside left normal form for the inverse braid monoid} $IB_n$
and the 
index $m$ we call the \emph{power} of $W$. The same way the
\emph{Garside right normal form for the inverse braid monoid} is defined and 
the corresponding
variant of Theorem~\ref{Theorem:garnfm} is true. 
\begin{Theorem} 
The necessary and sufficient condition that two words in  
$IB_{n}$ are equal is that their Garside normal forms are identical.
The Garside normal form  gives a solution to the word problem in the braid
group.
\label{Theorem:gws}
\end{Theorem}
\begin{proof}
As we noted in the proof of the previous Theorem the elements 
$\sigma_{i_1}\dots\sigma_{1}\dots \sigma_{i_k}
\dots\sigma_{k}$ and $ \sigma_{k}\dots\sigma_{j_k}\dots\sigma_{1}
\dots\sigma_{j_1}$ are uniquely determined. Also in  \cite{EL} 
(implicitly) there was given an algorithm how to obtain the form 
(\ref{eq:form_inv}) for an arbitrary word $W$ in the alphabet $A$.
Then combining it with the Garside algorithm we get a solution of the 
word problem for the inverse braid monoid. 
\end{proof}
 
Garside normal form for the braid groups was precised in the subsequent 
works of S.~I.~Adyan \cite{Ad},
 W.~Thurston \cite{E_Th}, E.~El-Rifai and H.~R.~Morton \cite{EM}. 
Namely, there was introduced 
the \emph{left-greedy form} (in the terminology of
W.~Thurston \cite{E_Th}) 
\begin{equation*}
\Delta^t A_1 \dots A_k,
\end{equation*}
where $A_i$ are the successive
possible longest \emph{fragments of the word} $\Delta$ (in the terminology 
of S.~I.~Adyan \cite{Ad}) or \emph{positive permutation braids} (in the 
terminology of E.~El-Rifai and H.~R.~Morton \cite{EM}). Certainly, the same 
way the \emph{right-greedy form} is defined. These greedy forms are
defined for the inverse braid monoid the same way.

Let us consider the elements $m\in IB_n$ satisfying  the equation:
\begin{equation}
\epsilon_i m = \epsilon_i.
\label{eq:imak}
\end{equation}
Geometrically this means that removing the string  (if it exists) that 
starts at the point with the number $i$ we get a trivial braid on the 
rest $n-1$ strings. It is equivalent to the condition
\begin{equation}
m \epsilon_{\tau (m)(i)}  = \epsilon_{\tau(m)(i)},
\label{eq:imak2}
\end{equation}
where $\tau$ is the canonical map to the symmetric monoid (\ref{eq:tauIBn}).
With the exception of $\epsilon_i$ itself all such elements belong to
$Br_n$. We call such braids as $i$\,-{\it Makanin} and denote the subgroup
of $i$\,-Makanin braids by $A_i$. The  subgroups $A_i$, $i=1, \dots, n$,
 are conjugate
\begin{equation} 
A_i = \sigma_{i-1}^{-1} \dots \sigma_1^{-1} A_1\sigma_1 \dots \sigma_{i-1}
\label{eq:Aimak}
\end{equation} 
free subgroups. The group $A_1$ is freely generated by the set 
$\{x_1, \dots, x_{n-1}\}$ \cite{Joh1}, where 
\begin{equation}
x_i = 
\sigma_{i-1}^{-1} \dots \sigma_1^{-1} \sigma_1^2\sigma_1 \dots \sigma_{i-1}.
\label{eq:freegmak}
\end{equation} 
The intersection of all subgroups of $i$\,-Makanin braids is the group of 
Makanin braids
\begin{equation} 
Mak_n =\cap_ {i=1}^{n} A_i.
\label{eq:mak}
\end{equation}
That is the same as $m\in Mak_n$ if and only if the equation (\ref{eq:imak})
holds for all $i$.

\section{Monoids of partial generalised braids\label{sec:gen}}

Construction of partial braids can be applied to various generalisations of 
braids, namely to those where geometric or diagrammatic construction of braids 
takes place.
Let $S_g$ be a surface of genus $g$ probably with boundary components and 
punctures. We consider partial braids lying in a layer between two such surfaces:
$S_g\times I$ and take a set of isotopy classes of such braids.   
We get a monoid of partial braid of a surface $S_g$,  denote it by $IB_n(S_g)$. 
An interesting case is when the surface is a sphere $S^2$. So our partial 
braids are lying in a layer between two concentric spheres. 
It was proved by O.~Zariski \cite{Za1} and then
rediscovered by E.~Fadell and J.~Van Buskirk \cite{FaV} that the braid group
of a sphere has a presentation with generators 
$\sigma_i$, $i=1, ..., n-1$, the same as for the classical braid group satisfying
the braid relations (\ref{eq:brelations})
and the following sphere relation: 
\begin{equation}
\sigma_1 \sigma_2 \dots \sigma_{n-2}\sigma_{n-1}^2\sigma_{n-2} \dots
\sigma_2\sigma_1 =1.
\label{eq:spherelation}
\end{equation}
\begin{Theorem} We get a presentation of the monoid $IB_n(S^2)$ if we add to
the presentation (\ref{eq:invbrelations}) or the presentation (\ref{eq:invbrelations2})
of  $IB_n$ the sphere relation (\ref{eq:spherelation}).
It is a factorisable inverse monoid.
\end{Theorem} 
\begin{proof} Essentially it is the same as for $IB_n$. 
Denote temporarily by $M_n$ the monoid defined by the presentation
and $IB_n(S^2)$ denotes the monoid of homotopy classes.
We already used that 
every word in the alphabet $A$ is congruent (using the relations
(\ref{eq:invbrelations}) to a word of the form (\ref{eq:form_inv}).
Now note that for the sphere inverse braid monoid the alphabet is the same
and relations for $IB_n$ are included into the set of relations for $IB_n(S^2)$.
As in \cite{EL} the evident map 
$$\Psi : M_n\to IB_n(S^2)$$
is defined and proved that it is onto.
Let us prove that $\Psi$ is a monomorphism. Suppose that for two words
$W_1, W_2 \in M_n$ we have
$$\Psi(W_1 = \Psi (W_2).$$
That means that the corresponding braids are isotopic.
Using relations 
(\ref{eq:invbrelations}) transform the words
$W_1, W_2$ into the form (\ref{eq:form_inv}) 
$$\sigma(i_1,\dots i_k; k) \epsilon_{k+1,n}x \epsilon_{k+1,n}
\sigma(k; j_1,\dots j_k).$$ 
Then  the corresponding fragments $\sigma(i_1,\dots i_k; k)$ and
$\sigma(k, j_1,\dots j_k; k)$ for $W_1$ and $W_2$ coincide.
The  elements $x_1$ of $W_1$ and $x_2 $ of $W_2$, which are the words 
on $\sigma_1, \dots, \sigma_k$, correspond after $\Psi$ to homotopic 
braids on $k$ strings on the sphere $S^2$. So $x_1$ can be transformed into $x_2$
using  relations for the braid groups $Br_k(S^2)$. The words $W_1$ and $W_2$
represent the same element in $M_n$. 
\end{proof}
Another example here is the braid group of a punctured disc which is 
isomorphic to the Artin-Brieskorn braid group of the type $B$ \cite{Bri1},
\cite{Ve6}. With respect to the classical braid group it has an extra generator
$\tau$ and the relations of type $B$:
\begin{equation}
\begin{cases} \tau\sigma_1\tau\sigma_1 &= \sigma_1\tau\sigma_1\tau, \\ 
\sigma_i \sigma_j &=\sigma_j \, \sigma_i, \ \
\text{if} \ \ |i-j|>1,
\end{cases} \label{eq:typbrel}
\end{equation}
Denote by $IBB_n$ the monoid of partial braids of the type $B$.
\begin{Theorem} We get a presentation of the monoid $IBB_n$ if we add to
the presentation (\ref{eq:invbrelations}) or the presentation (\ref{eq:invbrelations2})
of  $IB_n$ one generator $\tau$, the type $B$ relation
(\ref{eq:typbrel}) and the following relations
\begin{equation}
 \begin{cases} 
&\tau\tau^{-1}=\tau^{-1}\tau =1,  \\
&\epsilon_1\tau = \tau\epsilon_1 = \epsilon_1.
\end{cases} \label{eq:IBB}
\end{equation}
It is a factorisable inverse monoid.
\end{Theorem} 
\begin{proof} The same as for $IB_n$. 
\end{proof}

\begin{Remark} 
Theorem~\ref{eq:IBB} can be easily generalised for partial braids in 
handlebodies \cite{Ve1}.
\end{Remark} 

The same way as for $IB_n$ the notion of Makanin braids can be defined for 
any surface and we get $Mak_n(S_g)\subset IB_n(S_g)$. The group of 
Makanin braids for the sphere was used in the exact 
sequence~(\ref{eq:maks2}). 

Let $BP_n$ be the braid-permutation group of R.~Fenn,
R.~Rim\'anyi and C.~Rourke \cite{FRR2}.
It is defined  as a subgroup of $\operatorname{Aut} F_n$, generated
by both sets of the automorphisms $\sigma_i$ of (\ref{eq:autf})
and  $\xi_i$ of the following form: 
\begin{equation}
\begin{cases} x_i &\mapsto x_{i+1}, \\ x_{i+1} &\mapsto
x_i,     \\ x_j &\mapsto x_j, j\not=i,i+1, \end{cases}
\label{eq:syminaut}
\end{equation}
 R.~Fenn,
R.~Rim\'anyi and C.~Rourke proved that 
this group
is given by the set of generators: $\{ \xi_i, \sigma_i, \ \ i=1,2,
..., n-1 \}$ and relations: 
$$ \begin{cases} \xi_i^2&=1, \\ \xi_i \xi_j
&=\xi_j \xi_i, \ \ \text {if} \ \ |i-j| >1,\\ \xi_i \xi_{i+1}
\xi_i &= \xi_{i+1} \xi_i \xi_{i+1}. \end{cases} $$ 
\vglue 0.01cm
\centerline { The symmetric group relations} 
$$ \begin{cases} \sigma_i \sigma_j &=\sigma_j
\sigma_i, \ \text {if} \  |i-j| >1,
\\ \sigma_i \sigma_{i+1} \sigma_i &= \sigma_{i+1} \sigma_i \sigma_{i+1}.
\end{cases} $$
\vglue 0.01cm
\centerline {The braid group relations } 
\begin{equation}
 \begin{cases} \sigma_i \xi_j
&=\xi_j \sigma_i, \ \text {if} \  |i-j| >1,
\\ \xi_i \xi_{i+1} \sigma_i &= \sigma_{i+1} \xi_i \xi_{i+1},
\\ \sigma_i \sigma_{i+1} \xi_i &= \xi_{i+1} \sigma_i \sigma_{i+1}.
\end{cases} \label{eq:mixperm}
\end{equation}
\vglue0.01cm
\centerline {The mixed relations for the braid-permutation group}
\smallskip

R.~Fenn, R.~Rim\'anyi and C.~Rourke also gave a geometric
interpretation of $BP_n$ as a group of {\it welded braids}.

We consider the image of monoid $I_n$ in $\operatorname{End} F_n$
by the map  defined by the formulas (\ref{eq:syminaut}), (\ref{eq:endf1}).
We take also the  monoid $IB_n$  lying in $\operatorname{End} F_n$
under the map $\phi_n$ of Theorem~(\ref{theo:isoendo}). We define the 
{\it braid-permutation} 
monoid as a submonoid of $\operatorname{End} F_n$ generated by both
images of $IB_n$ and $I_n$ and denote it by $IBP_n$. It can be also defined
by the diagrams of partial welded braids. 
\begin{Theorem} We get a presentation of the monoid $IBP_n$ if we add to
the presentation of  $BP_n$ the generator $\epsilon$, relations 
(\ref{eq:invbrelations})
and the analogous relations between
$\xi_i$ and $\epsilon$, or 
generators $\epsilon_i$, $ 1\leq i \leq n$
relations (\ref{eq:invbrelations2}) and the analogous relations between
$\xi_i$ and $\epsilon_i$. 
It is a factorisable inverse monoid.
\end{Theorem} 
\begin{proof} The same as for $BP_n$. 
\end{proof}

The virtual braids \cite{Ve8} can be defined by the plane diagrams with real 
and virtual crossings. The corresponding Reidemeister moves are the same as 
for the welded braids of the  braid-permutation group with one exception.
 The forbidden move
corresponds to the last mixed relation for the braid-permutation group.
This allows to define the partial virtual braids and the corresponding
monoid $IVB_n$. So the mixed relation for $IVB_n$ have the form:
\begin{equation} \begin{cases} \sigma_i \xi_j
&=\xi_j \sigma_i, \ \text {if} \  |i-j| >1,
\\ \xi_i \xi_{i+1} \sigma_i &= \sigma_{i+1} \xi_i \xi_{i+1}.
\end{cases} \label{eq:mixvir}
\end{equation}
\vglue0.01cm
\centerline {The mixed relations for virtual braids}
\smallskip

\begin{Theorem} We get a presentation of the monoid $IVB_n$ if we delete
the last mixed relation in the presentation of  $IBP_n$, that is 
replace the relations (\ref{eq:mixperm}) by (\ref{eq:mixvir})  
It is a factorisable inverse monoid. The canonical epimorphism 
$$ IVB_n\to IBP_n$$
is evidently defined.
\end{Theorem} 

The {\it singular braid monoid} $SB_n$  or {\it Baez--Birman monoid} 
\cite{Bae}, \cite{Bir2} is defined as a monoid with generators 
$\sigma_i, \sigma_i^{-1}, x_i$, $i=1,\dots,n-1,$ and relations
\begin{equation}
\begin{cases}
&\sigma_i\sigma_j=\sigma_j\sigma_i, \ \text {if} \ \ |i-j| >1,\\
&x_ix_j=x_jx_i, \ \text {if} \ \ |i-j| >1,\\
&x_i\sigma_j=\sigma_j x_i, \
\text {if} \ \ |i-j| \not=1,\\
&\sigma_i \sigma_{i+1} \sigma_i = \sigma_{i+1} \sigma_i
\sigma_{i+1},\\
&\sigma_i \sigma_{i+1} x_i = x_{i+1} \sigma_i \sigma_{i+1},\\
&\sigma_{i+1} \sigma_ix_{i+1} = x_i \sigma_{i+1} \sigma_i,\\
&\sigma_i\sigma_i^{-1}=\sigma_i^{-1}\sigma_i =1.
\end{cases}
\label{eq:singrel}
\end{equation}
In pictures $\sigma_i$ corresponds to the canonical generator of 
the braid group and $x_i$ represents an intersection
of the $i$th and $(i+1)$st strand as in
Figure~\ref{fi:singene}.
\begin{figure}
\input singene
\caption{}\label{fi:singene}
\end{figure}
The singular braid monoid on two strings is isomorphic to $\Z\oplus\Z^+$.
The constructions of $SB_n$ is geometric, so we can easily get the analogous
monoid of partial singular braids $PSB_n$.
\begin{Theorem} We get a presentation of the monoid $PSB_n$ if we add to
the presentation of  $SB_n$ the  generators $\epsilon_i$, $ 1\leq i \leq n$,
relations (\ref{eq:invbrelations2}) and the analogous relations between
$x_i$ and $\epsilon_i$. 
\end{Theorem} 
\begin{proof} The same as for $BP_n$. 
\end{proof}
\begin{Remark} The monoid $PSB_n$  is not neither factorisable nor  inverse.
\end{Remark}

The construction of braid groups on graphs \cite{Ghr}, \cite{FarS} is geometrical
so, the same way as for the classical braid groups we can define {\emph {partial
braids on a graph }}$\Gamma$ and the {\emph {monoid of partial braids on a graph}}
$\Gamma$ which will be evidently inverse, so we call it as
{\emph {inverse braid monoid on the graph}} $\Gamma$ and we denote it as
$IB_n\Gamma$.

\section{Partial braids and braided monoidal categories\label{sec:cat}}

The system of braid groups $Br_n$ is equipped 
 with the standard pairings 
$$ \mu : Br_k \times Br_l \rightarrow Br_{k+l}. $$
It may be constructed by means of adding $l$ extra
strings to the initial $k$. If $ \sigma_i^\prime $ are the generators of
$Br_k $, $\sigma_j^{\prime\prime}$ are the generators of $Br_l$ and
$\sigma_r$ are the generators of $Br_{k+l}$, then the map $\mu$ can be
expressed in the form
$$ \mu (\sigma_i^\prime, e) = \sigma_i, \  1 \leq i \leq k-1, $$
$$\mu (e, \sigma_j^{\prime\prime}) = \sigma_{j+k}, \ 1 \leq j \leq l-1.$$
The same geometric construction allows to extend this pairing to a pairing for the
inverse braid monoids.
$$ \mu : IB_k \times IB_l \rightarrow IB_{k+l} $$
such that the following diagram commutes
\begin{equation}
\begin{CD}
Br_k \times Br_l @>\mu >>Br_{k+l}   \\
@VV\kappa V  @VV\kappa V \\
IB_k \times IB_l @>\mu >>IB_{k+l}.
\label{eq:comdpar} \end{CD}
\end{equation}
The vertical lines denote here the canonical inclusions.
For the generators $\epsilon_i$ we have:
$$ \mu (\epsilon_i^\prime, e) = \epsilon_i, \  1 \leq i \leq k, $$
$$\mu (e, \epsilon_j^{\prime\prime}) = \epsilon_{j+k}, \ 1 \leq j \leq l.$$

A strict monoidal (tensor) category
$\mathcal B$ is defined in a standard way. Its objects
$\{\overline 0, \overline 1, ...\}$
correspond to integers from $0$ to infinity and morphisms are defined by the
formula:
\begin{equation}
\operatorname{hom} (\overline k,\overline l)=\begin{cases} Br_k, \ \  &\text{if}
\quad  k=l,\\
\varnothing, \ \  &\text{if} \quad k\not=l. 
\end{cases} 
\end{equation}
The product in $\mathcal B$ is defined on objects by the sum of numbers and on
morphisms, by the pairing $\mu$. 
The category $\mathcal B$, generated by the braid groups, is a
{\it braided monoidal category} as defined by A.~Joyal and R.~Street \cite{JS}.

The following system of elements
$$\sigma_m \dots \sigma_{1}\sigma_{m+1}\dots\sigma_2\dots \sigma_{n+m-1}\dots
\sigma_n \in Br_{m+n}$$
defines a braiding $c$ in $\mathcal B$. Graphically it is depicted in
Figure~\ref{fi:braidin}. 
\begin{figure}
\input braidin
\caption{}\label{fi:braidin}
\end{figure}

The same way we define a strict monoidal category $\mathcal{IB}$ with the same
objects as for $\mathcal B$ and morphisms
\begin{equation}
\operatorname{hom} (\overline k,\overline l)=\begin{cases} IB_k, \ \  &\text{if}
\quad  k=l,\\
\varnothing, \ \  &\text{if} \quad k\not=l. 
\end{cases} 
\end{equation}

The canonical inclusions 
\begin{equation}
\kappa_n: Br_n\to IB_n
\end{equation}
define a functor 
$${\mathcal K}: \mathcal B\to{\mathcal{IB}}.$$

The image of the braiding $c$ by the functor
${\mathcal K}: \mathcal B\to{\mathcal{IB}}$
is a braiding in the category ${\mathcal{IB}}$. 

Geometrically the fact that the braiding for $\mathcal B$ defines also a braiding 
for partial braids is easily seen from Figure~\ref{fi:braidin}.
\begin{Proposition} The image of the braiding $c$ in the category
$\mathcal B$ by the functor $\mathcal K$ is a
braiding in the category $\mathcal{IB}$, so it becomes a braided 
monoidal category and
the functor $\mathcal K$ becomes a~morphism between the braided 
monoidal categories. 
\end{Proposition}
\begin{proof} By definition, the naturality of the braiding
$\mathcal K(c)$ (which we denote by the same symbol $c$)
means that the following equality
$$c_{\overline m,\overline n}\cdot\mu(b_m^\prime,b_n^{\prime\prime})=
\mu(b_n^{\prime\prime},b_m^\prime)\cdot c_{\overline m,\overline n}, \ \
b_m^\prime \in Br_m, \ \ b_n^{\prime\prime}\in Br_n,$$
is fulfilled. This amounts to the expression
$$c_{\overline m,\overline n}\cdot\mu(b_m^\prime,b_n^{\prime\prime})
\cdot c_{\overline m,\overline n}^{-1}=\mu(b_n^{\prime\prime},b_m^\prime),$$
which means that the
conjugation by the element $c_{\overline m,\overline n}$ transforms the
elements of $IB_m\times IB_n$, lying canonically in
$IB_{m+n}$, into the corresponding elements of $IB_n\times IB_m$.
The elements $c_{\overline m,\overline n}$ define a braiding for the
category $\mathcal B$, so, for checking the naturality of $c$ in
$\mathcal {IB}$
it remains to verify the naturality for the generators
$\epsilon_i,\ 1\leq i\leq m, \ m+1 \leq i \leq m+n$. Let us consider
the corresponding conjugation:
$$\sigma_m \dots \sigma_{1}\sigma_{m+1}\dots \sigma_2\dots \sigma_{n+m-1}\dots
\sigma_n \epsilon_i \sigma_n^{-1} \dots \sigma_{n+m-1}^{-1}\dots 
\sigma_{2}^{-1}\sigma_{m+1}^{-1}
\sigma_1^{-1}\dots \sigma_m^{-1}.$$
When $i>n$, we move $\epsilon_i$ back, using the relation
$$\epsilon_i\sigma_{i} = \sigma_i\epsilon_{i+1}.$$
We have:
$$\sigma_m ... \sigma_{1}\sigma_{m+1}...\sigma_2... \sigma_{n+m-1}...
\sigma_n \epsilon_i \sigma_n^{-1} ... \sigma_{n+m-1}^{-1}...
\sigma_{2}^{-1}\sigma_{m+1}^{-1}
\sigma_1^{-1}...\sigma_m^{-1}=$$
$$\sigma_m ... \sigma_{1}\sigma_{m+1}...\sigma_2... \sigma_{n+m-1}... 
\sigma_{i+1} \sigma_i \epsilon_{i-1}
\sigma_{i-1}\sigma_{i-1}^{-1} ... \sigma_{n+m-1}^{-1}...\sigma_{2}^{-1}\sigma_{m+1}^{-1}
\sigma_1^{-1}...\sigma_m^{-1}=$$
$$...=\epsilon_{i-n}.$$
When $i<n$, we move $\epsilon_i$ back using the relation
$$\epsilon_{i+1}\sigma_{i}=\sigma_{i}\epsilon_{i}.$$
We have:
$$\sigma_m ... \sigma_{1}\sigma_{m+1}...\sigma_2... \sigma_{n+m-1}...
\sigma_n \epsilon_i \sigma_n^{-1} ... \sigma_{n+m-1}^{-1}...
\sigma_{2}^{-1}\sigma_{m+1}^{-1}
\sigma_1^{-1}...\sigma_m^{-1}=$$
$$\sigma_m ... \sigma_{1}\sigma_{m+1}...\sigma_2... 
\sigma_i\sigma_{i+m-1}...\sigma_{i+2}
\sigma_{i+1}\sigma_i \epsilon_i \sigma_i^{-1}\sigma_{i+1}^{-1} ... 
\sigma_{i+m-1}^{-1}...
\sigma_{2}^{-1}\sigma_{m+1}^{-1}
\sigma_1^{-1}...\sigma_m^{-1}=$$
$$\sigma_m ... \sigma_{1}\sigma_{m+1}...\sigma_2... 
\sigma_i\sigma_{i+m-1}...\sigma_{i+2}
\sigma_{i+1} \epsilon_{i+1} \sigma_{i+1}^{-1} ... \sigma_{i+m-1}^{-1}\sigma_i^{-1}...
\sigma_{2}^{-1}\sigma_{m+1}^{-1}
\sigma_1^{-1}...\sigma_m^{-1}=$$
$$=\epsilon_{i+m}.$$
The conditions of coherence are fulfilled because they are true for $\mathcal B$.
\end{proof}

Let $BIB$  denote the classifying spaces of the limit inverse braid
monoid. As usual, the pairings
$\mu_{m,n}$ define a monoid structure on the disjoint sum of the classifying
spaces of $IB_n$:
$$\amalg_{n\geq 0} BIB_n.$$
\begin{Proposition} The canonical maps
$$BIB\to\Omega B(\amalg_{n\geq 0} BIB_n)$$
induce isomorphisms in homology
$$H_*(BIB;A)\to H_*((\Omega B(\amalg_{n\geq 0} BIB_n))_0; A),$$
with any (constant) coefficients. So,
$$BIB^+\cong (\Omega B(\amalg_{n\geq 0} BIB_n))_0 .$$
\end{Proposition}

The proof is the same as that of Theorem 3.2.1 and Corollary 3.2.2 in \cite{Ad2}
or (which is essentially the same) based directly on \cite{May2}. The braiding $c$
gives the necessary homotopy commutativity for the $H$-spaces
$\amalg_{n\geq 0} BIB_n$\,.

\begin{Theorem} The homomorphisms $\kappa_n$  induce morphisms
of braided monoidal categories
$$\mathcal B\buildrel\mathcal K\over\longrightarrow\mathcal{IB}$$
and the corresponding double loop maps
$$\Omega^2S^2\longrightarrow\Omega B(\amalg_{n\geq 0} BIB_n).$$
\end{Theorem}

The proof follows from the fact that the classifying space of a
braided monoidal category is a double loop space after group completion.

\end{document}

%% file: mapcl.tex
{\def\emline#1#2#3#4#5#6{%
       \put(#1,#2){\special{em:moveto}}%
       \put(#4,#5){\special{em:lineto}}}
\unitlength=16mm \special{em:linewidth 0.6pt}
\linethickness{0.6pt}

\begin{picture}(8.8,2.6)(-4.5,-1.3)
\emline{2.5}{1.}{1}{2.5598}{0.9404}{1}
\emline{2.5598}{0.9404}{1}{2.619}{0.88078}{1}
\emline{2.619}{0.88078}{1}{2.677}{0.82111}{1}
\emline{2.677}{0.82111}{1}{2.7332}{0.76137}{1}
\emline{2.7332}{0.76137}{1}{2.787}{0.70153}{1}
\emline{2.787}{0.70153}{1}{2.8379}{0.64156}{1}
\emline{2.8379}{0.64156}{1}{2.8853}{0.58145}{1}
\emline{2.8853}{0.58145}{1}{2.9284}{0.52117}{1}
\emline{2.9284}{0.52117}{1}{2.9669}{0.4607}{1}
\emline{2.9669}{0.4607}{1}{3.}{0.4}{1}
\emline{3.}{0.4}{1}{3.0273}{0.33916}{1}
\emline{3.0273}{0.33916}{1}{3.0486}{0.27869}{1}
\emline{3.0486}{0.27869}{1}{3.0639}{0.21917}{1}
\emline{3.0639}{0.21917}{1}{3.0732}{0.16122}{1}
\emline{3.0732}{0.16122}{1}{3.0764}{0.10542}{1}
\emline{3.0764}{0.10542}{1}{3.0735}{0.052399}{1}
\emline{3.0735}{0.052399}{1}{3.0644}{0.002742}{1}
\emline{3.0644}{0.002742}{1}{3.0492}{-0.042945}{1}
\emline{3.0492}{-0.042945}{1}{3.0277}{-0.084059}{1}
\emline{3.0277}{-0.084059}{1}{3.}{-0.12}{1}
\emline{3.}{-0.12}{1}{2.9662}{-0.15025}{1}
\emline{2.9662}{-0.15025}{1}{2.927}{-0.17465}{1}
\emline{2.927}{-0.17465}{1}{2.8834}{-0.19312}{1}
\emline{2.8834}{-0.19312}{1}{2.8364}{-0.20559}{1}
\emline{2.8364}{-0.20559}{1}{2.7868}{-0.21197}{1}
\emline{2.7868}{-0.21197}{1}{2.7358}{-0.2122}{1}
\emline{2.7358}{-0.2122}{1}{2.6841}{-0.20619}{1}
\emline{2.6841}{-0.20619}{1}{2.6328}{-0.19387}{1}
\emline{2.6328}{-0.19387}{1}{2.5828}{-0.17517}{1}
\emline{2.5828}{-0.17517}{1}{2.535}{-0.15}{1}
\emline{2.535}{-0.15}{1}{2.4902}{-0.11857}{1}
\emline{2.4902}{-0.11857}{1}{2.4481}{-0.082162}{1}
\emline{2.4481}{-0.082162}{1}{2.4082}{-0.042353}{1}
\emline{2.4082}{-0.042353}{1}{2.3701}{-0.00070547}{1}
\emline{2.3701}{-0.00070547}{1}{2.3332}{0.041212}{1}
\emline{2.3332}{0.041212}{1}{2.297}{0.081832}{1}
\emline{2.297}{0.081832}{1}{2.2611}{0.11959}{1}
\emline{2.2611}{0.11959}{1}{2.2249}{0.15291}{1}
\emline{2.2249}{0.15291}{1}{2.1881}{0.18024}{1}
\emline{2.1881}{0.18024}{1}{2.15}{0.2}{1}
\emline{2.15}{0.2}{1}{2.1105}{0.21099}{1}
\emline{2.1105}{0.21099}{1}{2.0707}{0.21346}{1}
\emline{2.0707}{0.21346}{1}{2.0318}{0.20802}{1}
\emline{2.0318}{0.20802}{1}{1.9953}{0.19529}{1}
\emline{1.9953}{0.19529}{1}{1.9624}{0.17587}{1}
\emline{1.9624}{0.17587}{1}{1.9345}{0.15039}{1}
\emline{1.9345}{0.15039}{1}{1.913}{0.11945}{1}
\emline{1.913}{0.11945}{1}{1.8992}{0.083658}{1}
\emline{1.8992}{0.083658}{1}{1.8944}{0.043639}{1}
\emline{1.8944}{0.043639}{1}{1.9}{0   }{1} \emline{1.9}{0
}{1}{1.9169}{-0.046595}{1}
\emline{1.9169}{-0.046595}{1}{1.9443}{-0.095279}{1}
\emline{1.9443}{-0.095279}{1}{1.9811}{-0.14514}{1}
\emline{1.9811}{-0.14514}{1}{2.026}{-0.19525}{1}
\emline{2.026}{-0.19525}{1}{2.0779}{-0.24471}{1}
\emline{2.0779}{-0.24471}{1}{2.1355}{-0.29258}{1}
\emline{2.1355}{-0.29258}{1}{2.1978}{-0.33797}{1}
\emline{2.1978}{-0.33797}{1}{2.2634}{-0.37995}{1}
\emline{2.2634}{-0.37995}{1}{2.3312}{-0.41759}{1}
\emline{2.3312}{-0.41759}{1}{2.4}{-0.45}{1}
\emline{2.4}{-0.45}{1}{2.4688}{-0.47641}{1}
\emline{2.4688}{-0.47641}{1}{2.5368}{-0.49674}{1}
\emline{2.5368}{-0.49674}{1}{2.6038}{-0.51107}{1}
\emline{2.6038}{-0.51107}{1}{2.6691}{-0.51948}{1}
\emline{2.6691}{-0.51948}{1}{2.7323}{-0.52205}{1}
\emline{2.7323}{-0.52205}{1}{2.7929}{-0.51885}{1}
\emline{2.7929}{-0.51885}{1}{2.8505}{-0.50997}{1}
\emline{2.8505}{-0.50997}{1}{2.9045}{-0.49548}{1}
\emline{2.9045}{-0.49548}{1}{2.9545}{-0.47546}{1}
\emline{2.9545}{-0.47546}{1}{3.}{-0.45}{1}
\emline{3.}{-0.45}{1}{3.0406}{-0.41926}{1}
\emline{3.0406}{-0.41926}{1}{3.0763}{-0.38375}{1}
\emline{3.0763}{-0.38375}{1}{3.1073}{-0.34407}{1}
\emline{3.1073}{-0.34407}{1}{3.1336}{-0.30082}{1}
\emline{3.1336}{-0.30082}{1}{3.1554}{-0.2546}{1}
\emline{3.1554}{-0.2546}{1}{3.1727}{-0.20602}{1}
\emline{3.1727}{-0.20602}{1}{3.1857}{-0.15567}{1}
\emline{3.1857}{-0.15567}{1}{3.1945}{-0.10415}{1}
\emline{3.1945}{-0.10415}{1}{3.1993}{-0.052057}{1}
\emline{3.1993}{-0.052057}{1}{3.2}{0   }{1} \emline{3.2}{0
}{1}{3.1969}{0.051539}{1}
\emline{3.1969}{0.051539}{1}{3.1899}{0.10253}{1}
\emline{3.1899}{0.10253}{1}{3.1791}{0.15304}{1}
\emline{3.1791}{0.15304}{1}{3.1646}{0.20316}{1}
\emline{3.1646}{0.20316}{1}{3.1463}{0.25296}{1}
\emline{3.1463}{0.25296}{1}{3.1243}{0.30253}{1}
\emline{3.1243}{0.30253}{1}{3.0986}{0.35193}{1}
\emline{3.0986}{0.35193}{1}{3.0693}{0.40126}{1}
\emline{3.0693}{0.40126}{1}{3.0365}{0.45059}{1}
\emline{3.0365}{0.45059}{1}{3.}{0.5}{1}
\emline{3.}{0.5}{1}{2.9601}{0.54955}{1}
\emline{2.9601}{0.54955}{1}{2.917}{0.59924}{1}
\emline{2.917}{0.59924}{1}{2.871}{0.64906}{1}
\emline{2.871}{0.64906}{1}{2.8226}{0.69899}{1}
\emline{2.8226}{0.69899}{1}{2.7721}{0.74901}{1}
\emline{2.7721}{0.74901}{1}{2.7198}{0.79912}{1}
\emline{2.7198}{0.79912}{1}{2.6661}{0.84928}{1}
\emline{2.6661}{0.84928}{1}{2.6113}{0.89949}{1}
\emline{2.6113}{0.89949}{1}{2.5558}{0.94974}{1}
\emline{2.5558}{0.94974}{1}{2.5}{1.}{1}
\emline{-2.5}{1.}{1}{-2.4911}{0.92263}{1}
\emline{-2.4911}{0.92263}{1}{-2.4821}{0.84574}{1}
\emline{-2.4821}{0.84574}{1}{-2.4732}{0.76978}{1}
\emline{-2.4732}{0.76978}{1}{-2.4642}{0.69523}{1}
\emline{-2.4642}{0.69523}{1}{-2.4552}{0.62257}{1}
\emline{-2.4552}{0.62257}{1}{-2.4462}{0.55226}{1}
\emline{-2.4462}{0.55226}{1}{-2.4371}{0.48478}{1}
\emline{-2.4371}{0.48478}{1}{-2.428}{0.42058}{1}
\emline{-2.428}{0.42058}{1}{-2.4189}{0.36012}{1}
\emline{-2.4189}{0.36012}{1}{-2.4093}{0.30338}{1}
\emline{-2.4093}{0.30338}{1}{-2.3989}{0.25011}{1}
\emline{-2.3989}{0.25011}{1}{-2.3871}{0.2}{1}
\emline{-2.3871}{0.2}{1}{-2.3735}{0.1528}{1}
\emline{-2.3735}{0.1528}{1}{-2.3576}{0.10822}{1}
\emline{-2.3576}{0.10822}{1}{-2.3388}{0.065969}{1}
\emline{-2.3388}{0.065969}{1}{-2.3168}{0.025784}{1}
\emline{-2.3168}{0.025784}{1}{-2.2909}{-0.012615}{1}
\emline{-2.2909}{-0.012615}{1}{-2.2612}{-0.049158}{1}
\emline{-2.2612}{-0.049158}{1}{-2.2283}{-0.083244}{1}
\emline{-2.2283}{-0.083244}{1}{-2.1928}{-0.11423}{1}
\emline{-2.1928}{-0.11423}{1}{-2.1554}{-0.14146}{1}
\emline{-2.1554}{-0.14146}{1}{-2.1168}{-0.1643}{1}
\emline{-2.1168}{-0.1643}{1}{-2.0775}{-0.1821}{1}
\emline{-2.0775}{-0.1821}{1}{-2.0384}{-0.19422}{1}
\emline{-2.0384}{-0.19422}{1}{-2.}{-0.2}{1}
\emline{-2.}{-0.2}{1}{-1.963}{-0.19895}{1}
\emline{-1.963}{-0.19895}{1}{-1.9281}{-0.19115}{1}
\emline{-1.9281}{-0.19115}{1}{-1.896}{-0.17684}{1}
\emline{-1.896}{-0.17684}{1}{-1.8674}{-0.15624}{1}
\emline{-1.8674}{-0.15624}{1}{-1.8429}{-0.12959}{1}
\emline{-1.8429}{-0.12959}{1}{-1.8232}{-0.097114}{1}
\emline{-1.8232}{-0.097114}{1}{-1.8091}{-0.05905}{1}
\emline{-1.8091}{-0.05905}{1}{-1.8012}{-0.015626}{1}
\emline{-1.8012}{-0.015626}{1}{-1.8}{0.0329}{1}
\emline{-1.8}{0.0329}{1}{-1.8058}{0.085969}{1}
\emline{-1.8058}{0.085969}{1}{-1.8183}{0.14283}{1}
\emline{-1.8183}{0.14283}{1}{-1.8371}{0.20271}{1}
\emline{-1.8371}{0.20271}{1}{-1.8619}{0.26487}{1}
\emline{-1.8619}{0.26487}{1}{-1.8925}{0.32855}{1}
\emline{-1.8925}{0.32855}{1}{-1.9285}{0.39298}{1}
\emline{-1.9285}{0.39298}{1}{-1.9698}{0.45742}{1}
\emline{-1.9698}{0.45742}{1}{-2.0159}{0.5211}{1}
\emline{-2.0159}{0.5211}{1}{-2.0665}{0.58361}{1}
\emline{-2.0665}{0.58361}{1}{-2.1211}{0.64507}{1}
\emline{-2.1211}{0.64507}{1}{-2.1791}{0.70561}{1}
\emline{-2.1791}{0.70561}{1}{-2.2399}{0.7654}{1}
\emline{-2.2399}{0.7654}{1}{-2.3029}{0.82458}{1}
\emline{-2.3029}{0.82458}{1}{-2.3677}{0.88331}{1}
\emline{-2.3677}{0.88331}{1}{-2.4336}{0.94173}{1}
\emline{-2.4336}{0.94173}{1}{-2.5}{1.}{1}
\emline{-2.5}{1.}{1}{-2.5089}{0.92263}{1}
\emline{-2.5089}{0.92263}{1}{-2.5179}{0.84574}{1}
\emline{-2.5179}{0.84574}{1}{-2.5268}{0.76978}{1}
\emline{-2.5268}{0.76978}{1}{-2.5358}{0.69523}{1}
\emline{-2.5358}{0.69523}{1}{-2.5448}{0.62257}{1}
\emline{-2.5448}{0.62257}{1}{-2.5538}{0.55226}{1}
\emline{-2.5538}{0.55226}{1}{-2.5629}{0.48478}{1}
\emline{-2.5629}{0.48478}{1}{-2.572}{0.42058}{1}
\emline{-2.572}{0.42058}{1}{-2.5811}{0.36012}{1}
\emline{-2.5811}{0.36012}{1}{-2.5907}{0.30338}{1}
\emline{-2.5907}{0.30338}{1}{-2.6011}{0.25011}{1}
\emline{-2.6011}{0.25011}{1}{-2.6129}{0.2}{1}
\emline{-2.6129}{0.2}{1}{-2.6265}{0.1528}{1}
\emline{-2.6265}{0.1528}{1}{-2.6424}{0.10822}{1}
\emline{-2.6424}{0.10822}{1}{-2.6612}{0.065969}{1}
\emline{-2.6612}{0.065969}{1}{-2.6832}{0.025784}{1}
\emline{-2.6832}{0.025784}{1}{-2.7091}{-0.012615}{1}
\emline{-2.7091}{-0.012615}{1}{-2.7388}{-0.049158}{1}
\emline{-2.7388}{-0.049158}{1}{-2.7717}{-0.083244}{1}
\emline{-2.7717}{-0.083244}{1}{-2.8072}{-0.11423}{1}
\emline{-2.8072}{-0.11423}{1}{-2.8446}{-0.14146}{1}
\emline{-2.8446}{-0.14146}{1}{-2.8832}{-0.1643}{1}
\emline{-2.8832}{-0.1643}{1}{-2.9225}{-0.1821}{1}
\emline{-2.9225}{-0.1821}{1}{-2.9616}{-0.19422}{1}
\emline{-2.9616}{-0.19422}{1}{-3.}{-0.2}{1}
\emline{-3.}{-0.2}{1}{-3.037}{-0.19895}{1}
\emline{-3.037}{-0.19895}{1}{-3.0719}{-0.19115}{1}
\emline{-3.0719}{-0.19115}{1}{-3.104}{-0.17684}{1}
\emline{-3.104}{-0.17684}{1}{-3.1326}{-0.15624}{1}
\emline{-3.1326}{-0.15624}{1}{-3.1571}{-0.12959}{1}
\emline{-3.1571}{-0.12959}{1}{-3.1768}{-0.097114}{1}
\emline{-3.1768}{-0.097114}{1}{-3.1909}{-0.05905}{1}
\emline{-3.1909}{-0.05905}{1}{-3.1988}{-0.015626}{1}
\emline{-3.1988}{-0.015626}{1}{-3.2}{0.0329}{1}
\emline{-3.2}{0.0329}{1}{-3.1942}{0.085969}{1}
\emline{-3.1942}{0.085969}{1}{-3.1817}{0.14283}{1}
\emline{-3.1817}{0.14283}{1}{-3.1629}{0.20271}{1}
\emline{-3.1629}{0.20271}{1}{-3.1381}{0.26487}{1}
\emline{-3.1381}{0.26487}{1}{-3.1075}{0.32855}{1}
\emline{-3.1075}{0.32855}{1}{-3.0715}{0.39298}{1}
\emline{-3.0715}{0.39298}{1}{-3.0302}{0.45742}{1}
\emline{-3.0302}{0.45742}{1}{-2.9841}{0.5211}{1}
\emline{-2.9841}{0.5211}{1}{-2.9335}{0.58361}{1}
\emline{-2.9335}{0.58361}{1}{-2.8789}{0.64507}{1}
\emline{-2.8789}{0.64507}{1}{-2.8209}{0.70561}{1}
\emline{-2.8209}{0.70561}{1}{-2.7601}{0.7654}{1}
\emline{-2.7601}{0.7654}{1}{-2.6971}{0.82458}{1}
\emline{-2.6971}{0.82458}{1}{-2.6323}{0.88331}{1}
\emline{-2.6323}{0.88331}{1}{-2.5664}{0.94173}{1}
\emline{-2.5664}{0.94173}{1}{-2.5}{1.}{1} \special{em:linewidth
0.4pt}
\emline{-1.5}{0}{1}{-1.5079}{0.12533}{1}
\emline{-1.5079}{0.12533}{1}{-1.5314}{0.24869}{1}
\emline{-1.5314}{0.24869}{1}{-1.5702}{0.36812}{1}
\emline{-1.5702}{0.36812}{1}{-1.6237}{0.48175}{1}
\emline{-1.6237}{0.48175}{1}{-1.691}{0.58779}{1}
\emline{-1.691}{0.58779}{1}{-1.771}{0.68455}{1}
\emline{-1.771}{0.68455}{1}{-1.8626}{0.77051}{1}
\emline{-1.8626}{0.77051}{1}{-1.9642}{0.84433}{1}
\emline{-1.9642}{0.84433}{1}{-2.0742}{0.90483}{1}
\emline{-2.0742}{0.90483}{1}{-2.191}{0.95106}{1}
\emline{-2.191}{0.95106}{1}{-2.3126}{0.98229}{1}
\emline{-2.3126}{0.98229}{1}{-2.4372}{0.99803}{1}
\emline{-2.4372}{0.99803}{1}{-2.5628}{0.99803}{1}
\emline{-2.5628}{0.99803}{1}{-2.6874}{0.98229}{1}
\emline{-2.6874}{0.98229}{1}{-2.809}{0.95106}{1}
\emline{-2.809}{0.95106}{1}{-2.9258}{0.90483}{1}
\emline{-2.9258}{0.90483}{1}{-3.0358}{0.84433}{1}
\emline{-3.0358}{0.84433}{1}{-3.1374}{0.77051}{1}
\emline{-3.1374}{0.77051}{1}{-3.229}{0.68455}{1}
\emline{-3.229}{0.68455}{1}{-3.309}{0.58779}{1}
\emline{-3.309}{0.58779}{1}{-3.3763}{0.48175}{1}
\emline{-3.3763}{0.48175}{1}{-3.4298}{0.36812}{1}
\emline{-3.4298}{0.36812}{1}{-3.4686}{0.24869}{1}
\emline{-3.4686}{0.24869}{1}{-3.4921}{0.12533}{1}
\emline{-3.4921}{0.12533}{1}{-3.5}{0}{1}
\emline{-3.5}{0}{1}{-3.4921}{-0.12533}{1}
\emline{-3.4921}{-0.12533}{1}{-3.4686}{-0.24869}{1}
\emline{-3.4686}{-0.24869}{1}{-3.4298}{-0.36812}{1}
\emline{-3.4298}{-0.36812}{1}{-3.3763}{-0.48175}{1}
\emline{-3.3763}{-0.48175}{1}{-3.309}{-0.58779}{1}
\emline{-3.309}{-0.58779}{1}{-3.229}{-0.68455}{1}
\emline{-3.229}{-0.68455}{1}{-3.1374}{-0.77051}{1}
\emline{-3.1374}{-0.77051}{1}{-3.0358}{-0.84433}{1}
\emline{-3.0358}{-0.84433}{1}{-2.9258}{-0.90483}{1}
\emline{-2.9258}{-0.90483}{1}{-2.809}{-0.95106}{1}
\emline{-2.809}{-0.95106}{1}{-2.6874}{-0.98229}{1}
\emline{-2.6874}{-0.98229}{1}{-2.5628}{-0.99803}{1}
\emline{-2.5628}{-0.99803}{1}{-2.4372}{-0.99803}{1}
\emline{-2.4372}{-0.99803}{1}{-2.3126}{-0.98229}{1}
\emline{-2.3126}{-0.98229}{1}{-2.191}{-0.95106}{1}
\emline{-2.191}{-0.95106}{1}{-2.0742}{-0.90483}{1}
\emline{-2.0742}{-0.90483}{1}{-1.9642}{-0.84433}{1}
\emline{-1.9642}{-0.84433}{1}{-1.8626}{-0.77051}{1}
\emline{-1.8626}{-0.77051}{1}{-1.771}{-0.68455}{1}
\emline{-1.771}{-0.68455}{1}{-1.691}{-0.58779}{1}
\emline{-1.691}{-0.58779}{1}{-1.6237}{-0.48175}{1}
\emline{-1.6237}{-0.48175}{1}{-1.5702}{-0.36812}{1}
\emline{-1.5702}{-0.36812}{1}{-1.5314}{-0.24869}{1}
\emline{-1.5314}{-0.24869}{1}{-1.5079}{-0.12533}{1}
\emline{-1.5079}{-0.12533}{1}{-1.5}{0}{1}
\emline{1.}{0}{1}{0.99211}{0.12533}{1}
\emline{0.99211}{0.12533}{1}{0.96858}{0.24869}{1}
\emline{0.96858}{0.24869}{1}{0.92978}{0.36812}{1}
\emline{0.92978}{0.36812}{1}{0.87631}{0.48175}{1}
\emline{0.87631}{0.48175}{1}{0.80902}{0.58779}{1}
\emline{0.80902}{0.58779}{1}{0.72897}{0.68455}{1}
\emline{0.72897}{0.68455}{1}{0.63742}{0.77051}{1}
\emline{0.63742}{0.77051}{1}{0.53583}{0.84433}{1}
\emline{0.53583}{0.84433}{1}{0.42578}{0.90483}{1}
\emline{0.42578}{0.90483}{1}{0.30902}{0.95106}{1}
\emline{0.30902}{0.95106}{1}{0.18738}{0.98229}{1}
\emline{0.18738}{0.98229}{1}{0.062791}{0.99803}{1}
\emline{0.062791}{0.99803}{1}{-0.062791}{0.99803}{1}
\emline{-0.062791}{0.99803}{1}{-0.18738}{0.98229}{1}
\emline{-0.18738}{0.98229}{1}{-0.30902}{0.95106}{1}
\emline{-0.30902}{0.95106}{1}{-0.42578}{0.90483}{1}
\emline{-0.42578}{0.90483}{1}{-0.53583}{0.84433}{1}
\emline{-0.53583}{0.84433}{1}{-0.63742}{0.77051}{1}
\emline{-0.63742}{0.77051}{1}{-0.72897}{0.68455}{1}
\emline{-0.72897}{0.68455}{1}{-0.80902}{0.58779}{1}
\emline{-0.80902}{0.58779}{1}{-0.87631}{0.48175}{1}
\emline{-0.87631}{0.48175}{1}{-0.92978}{0.36812}{1}
\emline{-0.92978}{0.36812}{1}{-0.96858}{0.24869}{1}
\emline{-0.96858}{0.24869}{1}{-0.99211}{0.12533}{1}
\emline{-0.99211}{0.12533}{1}{-1.}{0}{1}
\emline{-1.}{0}{1}{-0.99211}{-0.12533}{1}
\emline{-0.99211}{-0.12533}{1}{-0.96858}{-0.24869}{1}
\emline{-0.96858}{-0.24869}{1}{-0.92978}{-0.36812}{1}
\emline{-0.92978}{-0.36812}{1}{-0.87631}{-0.48175}{1}
\emline{-0.87631}{-0.48175}{1}{-0.80902}{-0.58779}{1}
\emline{-0.80902}{-0.58779}{1}{-0.72897}{-0.68455}{1}
\emline{-0.72897}{-0.68455}{1}{-0.63742}{-0.77051}{1}
\emline{-0.63742}{-0.77051}{1}{-0.53583}{-0.84433}{1}
\emline{-0.53583}{-0.84433}{1}{-0.42578}{-0.90483}{1}
\emline{-0.42578}{-0.90483}{1}{-0.30902}{-0.95106}{1}
\emline{-0.30902}{-0.95106}{1}{-0.18738}{-0.98229}{1}
\emline{-0.18738}{-0.98229}{1}{-0.062791}{-0.99803}{1}
\emline{-0.062791}{-0.99803}{1}{0.062791}{-0.99803}{1}
\emline{0.062791}{-0.99803}{1}{0.18738}{-0.98229}{1}
\emline{0.18738}{-0.98229}{1}{0.30902}{-0.95106}{1}
\emline{0.30902}{-0.95106}{1}{0.42578}{-0.90483}{1}
\emline{0.42578}{-0.90483}{1}{0.53583}{-0.84433}{1}
\emline{0.53583}{-0.84433}{1}{0.63742}{-0.77051}{1}
\emline{0.63742}{-0.77051}{1}{0.72897}{-0.68455}{1}
\emline{0.72897}{-0.68455}{1}{0.80902}{-0.58779}{1}
\emline{0.80902}{-0.58779}{1}{0.87631}{-0.48175}{1}
\emline{0.87631}{-0.48175}{1}{0.92978}{-0.36812}{1}
\emline{0.92978}{-0.36812}{1}{0.96858}{-0.24869}{1}
\emline{0.96858}{-0.24869}{1}{0.99211}{-0.12533}{1}
\emline{0.99211}{-0.12533}{1}{1.}{0}{1}
\emline{3.5}{0}{1}{3.4921}{0.12533}{1}
\emline{3.4921}{0.12533}{1}{3.4686}{0.24869}{1}
\emline{3.4686}{0.24869}{1}{3.4298}{0.36812}{1}
\emline{3.4298}{0.36812}{1}{3.3763}{0.48175}{1}
\emline{3.3763}{0.48175}{1}{3.309}{0.58779}{1}
\emline{3.309}{0.58779}{1}{3.229}{0.68455}{1}
\emline{3.229}{0.68455}{1}{3.1374}{0.77051}{1}
\emline{3.1374}{0.77051}{1}{3.0358}{0.84433}{1}
\emline{3.0358}{0.84433}{1}{2.9258}{0.90483}{1}
\emline{2.9258}{0.90483}{1}{2.809}{0.95106}{1}
\emline{2.809}{0.95106}{1}{2.6874}{0.98229}{1}
\emline{2.6874}{0.98229}{1}{2.5628}{0.99803}{1}
\emline{2.5628}{0.99803}{1}{2.4372}{0.99803}{1}
\emline{2.4372}{0.99803}{1}{2.3126}{0.98229}{1}
\emline{2.3126}{0.98229}{1}{2.191}{0.95106}{1}
\emline{2.191}{0.95106}{1}{2.0742}{0.90483}{1}
\emline{2.0742}{0.90483}{1}{1.9642}{0.84433}{1}
\emline{1.9642}{0.84433}{1}{1.8626}{0.77051}{1}
\emline{1.8626}{0.77051}{1}{1.771}{0.68455}{1}
\emline{1.771}{0.68455}{1}{1.691}{0.58779}{1}
\emline{1.691}{0.58779}{1}{1.6237}{0.48175}{1}
\emline{1.6237}{0.48175}{1}{1.5702}{0.36812}{1}
\emline{1.5702}{0.36812}{1}{1.5314}{0.24869}{1}
\emline{1.5314}{0.24869}{1}{1.5079}{0.12533}{1}
\emline{1.5079}{0.12533}{1}{1.5}{0}{1}
\emline{1.5}{0}{1}{1.5079}{-0.12533}{1}
\emline{1.5079}{-0.12533}{1}{1.5314}{-0.24869}{1}
\emline{1.5314}{-0.24869}{1}{1.5702}{-0.36812}{1}
\emline{1.5702}{-0.36812}{1}{1.6237}{-0.48175}{1}
\emline{1.6237}{-0.48175}{1}{1.691}{-0.58779}{1}
\emline{1.691}{-0.58779}{1}{1.771}{-0.68455}{1}
\emline{1.771}{-0.68455}{1}{1.8626}{-0.77051}{1}
\emline{1.8626}{-0.77051}{1}{1.9642}{-0.84433}{1}
\emline{1.9642}{-0.84433}{1}{2.0742}{-0.90483}{1}
\emline{2.0742}{-0.90483}{1}{2.191}{-0.95106}{1}
\emline{2.191}{-0.95106}{1}{2.3126}{-0.98229}{1}
\emline{2.3126}{-0.98229}{1}{2.4372}{-0.99803}{1}
\emline{2.4372}{-0.99803}{1}{2.5628}{-0.99803}{1}
\emline{2.5628}{-0.99803}{1}{2.6874}{-0.98229}{1}
\emline{2.6874}{-0.98229}{1}{2.809}{-0.95106}{1}
\emline{2.809}{-0.95106}{1}{2.9258}{-0.90483}{1}
\emline{2.9258}{-0.90483}{1}{3.0358}{-0.84433}{1}
\emline{3.0358}{-0.84433}{1}{3.1374}{-0.77051}{1}
\emline{3.1374}{-0.77051}{1}{3.229}{-0.68455}{1}
\emline{3.229}{-0.68455}{1}{3.309}{-0.58779}{1}
\emline{3.309}{-0.58779}{1}{3.3763}{-0.48175}{1}
\emline{3.3763}{-0.48175}{1}{3.4298}{-0.36812}{1}
\emline{3.4298}{-0.36812}{1}{3.4686}{-0.24869}{1}
\emline{3.4686}{-0.24869}{1}{3.4921}{-0.12533}{1}
\emline{3.4921}{-0.12533}{1}{3.5}{0}{1}
\emline{0.34271}{0.044552}{1}{0.29692}{0.066888}{1}
\emline{0.29692}{0.066888}{1}{0.24967}{0.085976}{1}
\emline{0.24967}{0.085976}{1}{0.20122}{0.10172}{1}
\emline{0.20122}{0.10172}{1}{0.15178}{0.11405}{1}
\emline{0.15178}{0.11405}{1}{0.1016}{0.1229}{1}
\emline{0.1016}{0.1229}{1}{0.050922}{0.12822}{1}
\emline{0.050922}{0.12822}{1}{0}{0.13}{1}
\emline{0}{0.13}{1}{-0.050922}{0.12822}{1}
\emline{-0.050922}{0.12822}{1}{-0.1016}{0.1229}{1}
\emline{-0.1016}{0.1229}{1}{-0.15178}{0.11405}{1}
\emline{-0.15178}{0.11405}{1}{-0.20122}{0.10172}{1}
\emline{-0.20122}{0.10172}{1}{-0.24967}{0.085976}{1}
\emline{-0.24967}{0.085976}{1}{-0.29692}{0.066888}{1}
\emline{-0.29692}{0.066888}{1}{-0.34271}{0.044552}{1}
\emline{0.34271}{-0.044552}{1}{0.29692}{-0.066888}{1}
\emline{0.29692}{-0.066888}{1}{0.24967}{-0.085976}{1}
\emline{0.24967}{-0.085976}{1}{0.20122}{-0.10172}{1}
\emline{0.20122}{-0.10172}{1}{0.15178}{-0.11405}{1}
\emline{0.15178}{-0.11405}{1}{0.1016}{-0.1229}{1}
\emline{0.1016}{-0.1229}{1}{0.050922}{-0.12822}{1}
\emline{0.050922}{-0.12822}{1}{0}{-0.13}{1}
\emline{0}{-0.13}{1}{-0.050922}{-0.12822}{1}
\emline{-0.050922}{-0.12822}{1}{-0.1016}{-0.1229}{1}
\emline{-0.1016}{-0.1229}{1}{-0.15178}{-0.11405}{1}
\emline{-0.15178}{-0.11405}{1}{-0.20122}{-0.10172}{1}
\emline{-0.20122}{-0.10172}{1}{-0.24967}{-0.085976}{1}
\emline{-0.24967}{-0.085976}{1}{-0.29692}{-0.066888}{1}
\emline{-0.29692}{-0.066888}{1}{-0.34271}{-0.044552}{1}
%
\put(-3.28,-.1){.} \put(-1.57,0){\circle*{.05}} \put(-1.62,-.1){.}
\put(-1.68,-.1){.} \put(-1.74,-.1){.} \put(-.93,0){\circle*{.05}}
\put(-.8,-.1){.} \put(-.7,-.1){.} \put(-.6,-.1){.}
\put(.93,0){\circle*{.05}} \put(.8,-.1){.} \put(.7,-.1){.}
\put(.6,-.1){.} \put(1.57,0){\circle*{.05}} \put(1.62,-.1){.}
\put(1.68,-.1){.} \put(1.74,-.1){.} \put(3.43,0){\circle*{.05}}
\put(3.34,-.1){.} \put(3.28,-.1){.} \put(.4,0){\circle*{.05}}
\put(-.4,0){\circle*{.05}} \put(2.9,0){\circle*{.05}}
\put(2.1,0){\circle*{.05}} \put(-2.1,0){\circle*{.05}}
\put(-2.9,0){\circle*{.05}}
\put(-3,-.37){{$x_i$}} \put(-2.1,-.37){{$x_{i+1}$}}
\put(1.8,.32){{$x_{i+1}^{-1}x_ix_{i+1}$}}
%
%
%
\emline{0.34271}{0.044552}{1}{0.28}{0.12}{1}
\emline{0.34271}{0.044552}{1}{0.27}{0.042}{1}
\emline{-0.34271}{-0.044552}{1}{-0.28}{-0.12}{1}
\emline{-0.34271}{-0.044552}{1}{-0.27}{-0.042}{1}
%
\emline{-3.0302}{ 0.45742}{1}{-2.9479}{ 0.52163}{1}
\emline{-3.0302}{ 0.45742}{1}{-2.9968}{ 0.55634}{1}
\emline{2.4688}{-0.4768}{1}{2.3865}{-0.41255}{1}
\emline{2.4688}{-0.4768}{1}{2.3647}{-0.46846}{1}
\emline{-2.}{0.5}{1}{-1.9666}{0.40108}{1}
\emline{-2.}{0.5}{1}{-1.9177}{0.43579}{1}
\end{picture}
}

%% file: singene.tex
\begin{picture}(0,130)(0,-10) 
\thicklines
\put(0,50){\circle*{5}} \put(-100,100){\line(0,-1){100}}
\put(-50,100){\line(0,-1){100}} \put(-25,100){\line(1,-2){50}}
\put(25,100){\line(-1,-2){50}} \put(50,100){\line(0,-1){100}}
\put(100,100){\line(0,-1){100}}
\put(-100,110){\makebox(0,0)[cc]{$1$}}
\put(-50,110){\makebox(0,0)[cc]{$i-1$}}
\put(-25,110){\makebox(0,0)[cc]{$i$}}
\put(25,110){\makebox(0,0)[cc]{$i+1$}}
\put(50,110){\makebox(0,0)[cc]{$i+2$}}
\put(100,110){\makebox(0,0)[cc]{$n$}}
\put(-75,50){\makebox(0,0)[cc]{.\quad.\quad.}}
\put(75,50){\makebox(0,0)[cc]{.\quad.\quad.}}
\end{picture}

%% file: braidin.tex
\begin{picture}(0,140)(0,-15)
\thicklines
\put(-50,100){\line(0,-1){50}}
\put(-25,100){\line(0,-1){25}}
\put(0,100){\line(1,-1){50}}
\put(25,100){\line(-1,-1){10}}
\put(50,75){\line(-1,-1){10}}
\put(50,100){\line(0,-1){25}}
\put(-50,50){\line(1,-1){50}}
\put(-25,75){\line(1,-1){50}}
\put(0,100){\line(1,-1){50}}
\put(50,50){\line(0,-1){50}}
\put(-50,25){\line(0,-1){25}}
\put(-35,40){\line(1,1){20}}
\put(25,25){\line(0,-1){25}}
\put(-10,65){\line(1,1){19}}
\put(15,40){\line(1,1){19}}
\put(-50,25){\line(1,1){10}}
\put(-25,0){\line(1,1){10}}
\put(-10,15){\line(1,1){19}}
\put(-25,110){\makebox(0,0)[cc]{$\overbrace{\quad \ \ \ \ \ \ \ \ \ \ \ \ \
\ \ }$}}
\put(-25,120){\makebox(0,0)[cc]{$m$}}
\put(37,110){\makebox(0,0)[cc]{$\overbrace{ \ \ \ \ \ \ \ \ \ \ }$}}
\put(37,120){\makebox(0,0)[cc]{$n$}}
\end{picture}
\vglue-0.3 cm